\documentclass[aos]{imsart}
\usepackage{xr}
\RequirePackage{amsthm,amsmath,amsfonts,amssymb}
\RequirePackage[numbers]{natbib}
\RequirePackage[colorlinks,citecolor=blue,urlcolor=blue]{hyperref}
\RequirePackage{graphicx}
\usepackage{graphicx}
\usepackage{caption,comment}
\usepackage{subcaption}
\usepackage{csquotes}
\usepackage{bm}

\startlocaldefs
\theoremstyle{plain}

\theoremstyle{remark}

\newcommand{\cM}{\mathcal{M}}
\newcommand{\R}{\mathbb{R}}
\newcommand{\E}{\mathbb{E}}
\newcommand{\Z}{\mathbb{Z}}
\newcommand{\N}{\mathbb{N}}

	\renewcommand{\P}{\mathbb{P}}

	\renewcommand{\H}{\mathbb{H}}
\newtheorem{thm}{Theorem}[section]

\newtheorem{lem}{Lemma}[section]
\newtheorem{ppn}[thm]{Proposition}

\theoremstyle{definition}
\newtheorem{defn}{Definition}[section]
\newtheorem{remark}{Remark}[section]
\DeclareMathOperator{\Var}{Var}
\endlocaldefs

\begin{document}

\begin{frontmatter}
\title{Inference in Ising models on dense regular graphs}
\runtitle{Inference in Ising models on dense regular graphs}

\begin{aug}
\author[A]{\fnms{Yuanzhe}~\snm{Xu}\ead[label=e1]{yuanzhe.xu@columbia.edu}}
\and
\author[B]{\fnms{Sumit}~\snm{Mukherjee}\ead[label=e2]{sm3949@columbia.edu}}
\address[A]{Department of Statistics,
Columbia University\printead[presep={,\ }]{e1}}
\address[B]{Department of Statistics,
Columbia University\printead[presep={,\ }]{e2}}
\end{aug}

\begin{abstract}
In this paper, we derive the limit of experiments for one parameter Ising models on dense regular graphs. In particular, we show that the limiting experiment is Gaussian in the \enquote{low temperature} regime, and non Gaussian in the \enquote{critical} regime. We also derive the limiting distributions of the maximum likelihood and maximum pseudo-likelihood estimators, and study limiting power for tests of hypothesis against contiguous alternatives. To the best of our knowledge, this is the first attempt at establishing the classical limits of experiments for Ising models (and more generally, Markov random fields).
\end{abstract}



\end{frontmatter}

\section{Introduction}
The Ising Model is possibly the most well-known discrete graphical model, originating in statistical physics (see \cite{harris1974effect, ising1925beitrag, onsager1944crystal}), and henceforth studied in depth across several disciplines, including statistics and machine learning (c.f.~\cite{berthet2019exact, bhattacharya2018inference, deb2020detecting, deb2020fluctuations,  ghosal2020joint, mukherjee2018global} and the references therein). Under this model, we observe a vector of dependent Rademacher (i.e.~$\pm1$ valued) random variables, where the dependency is controlled by a coupling matrix, and an \enquote{inverse temperature} parameter $\theta> 0$ (borrowing statistical physics terminology).  Very often this coupling matrix is taken to be the (scaled) adjacency matrix of a graph. Some of the common graph ensembles on which the Ising model has been studied include the complete graph (the corresponding Ising model is known as the Curie-Weiss model), the $d$ dimensional grid, Erdos-R\'enyi graphs, and random regular graphs (\cite{dembo2010ising, ellis1978statistics, kabluchko2019fluctuations, onsager1944crystal}). Note that all the graph ensembles in the above list are (approximately) regular graphs. 
\\

In this paper, we will study the behavior of Ising models on a sequence of \enquote{dense} regular graphs converging in cut metric (see section \ref{sec:graphon} for a brief introduction to the theory of dense graphs/graphons).
Given an Ising model on a dense regular graph parametrized by the inverse temperature parameter $\theta> 0$ (see~\eqref{Ising}), we study limits of experiments in the sense of Lucien Le Cam (\cite{le1972limits}, see also \cite{van2000asymptotic}). In particular, we show that the Ising model is locally asymptotically normal (LAN) in the low temperature regime ($\theta>1$), whereas the limiting experiment is very different in the critical ($\theta=1$) temperature regime. Using this framework, we derive the limiting power of tests involving the parameter $\theta$, based on the maximum likelihood estimate, pseudo-likelihood estimate, and the sample mean, across all regimes of $\theta$. We also study asymptotic limiting distributions of the maximum likelihood estimate and pseudo-likelihood estimate in the regime  $\theta\ge 1$ (where consistent estimation of $\theta$ is possible), and compare their asymptotic performances. Prior to our work, limit distribution of the maximum likelihood estimate was known only for the Curie-Weiss model (\cite{comets1991asymptotics}), and limit distribution for the pseudo-likelihood estimator was not known in any example (to the best of our knowledge). Thus, we give a complete toolbox for inference regarding the parameter $\theta$, for Ising model on dense graphs.


\section{Main results}
In this section we formally introduce the Ising model (section \ref{sec:formal}), and state our main results (section \ref{Results}), which are essentially of three types, (i) limits of experiments, (ii) asymptotic performance of estimators, and (iii) asymptotic performance of tests of hypothesis. We recall the notion of limits of experiments in section \ref{sec:limits}. To obtain convergence in experiments, we require the sequence of coupling matrices for the Ising model to converge in cut metric, a notion which we recall in section \ref{sec:graphon}. We also introduce the estimators and test statistics that we study in sections \ref{sec:est} and \ref{sec:ht} respectively. Section \ref{sec:example} illustrates our results with two concrete examples. Finally, section \ref{sec:222} discusses the main contributions of this paper, and possible avenues of future research.

\subsection{Formal set up}\label{sec:formal}

Let $n$ be a positive integer, and let $Q_n$ be a (known) symmetric $n\times n$ matrix with non-negative entries, and $0$ on the diagonal. Let $\theta\ge 0$ be an unknown real valued parameter. Then the Ising model with inverse temperature parameter $\theta$ and coupling matrix $Q_n$ is a probability distribution on $\{-1,1\}^n$, defined by the probability mass function
\begin{equation}\label{Ising}
\P_{\theta,Q_{n}}(\mathbf{X}=\mathbf{x}):=\exp{\Big(\frac{\theta}{2} \mathbf{x}^{T}Q_{n}\mathbf{x} - Z_{n}(\theta,Q_{n})\Big)}, \text{ for }{\mathbf x}\in \{-1,1\}^n.
\end{equation}

Here $Z(\theta,Q_n)$ is the log normalizing constant which makes \eqref{Ising} into a probability distribution. One of the major challenges in analyzing Ising models is that $Z_n(\theta,Q_n)$ is typically intractable, both analytically and computationally. Consequently, computing and analyzing the maximum likelihood estimate is extremely challenging. 

Throughout the paper we will assume that the matrix $Q_n$ satisfies the following assumptions:
\begin{itemize}
  \item The matrix $Q_{n}$ is regular, i.e.~setting $[n]:=\{1,2,\ldots,n\}$ we have
 \begin{align}
 \label{RIC}\sum_{j=1}^nQ_n(i,j)=1,\text{ for all }i\in [n].
 \end{align}
 \item
 There exists a finite positive constant $C_w$ free of $n$ such that
 \begin{align}
\label{CWC}
    \max_{i,j\in [n]}Q_n(i,j)<\frac{C_{w}}{n}.
\end{align}

\item

The Frobenius norm of $Q_n$ converges, i.e.
\begin{equation}\label{eq:frob}
    ||Q_{n}||_{F}:=\sqrt{\sum_{i,j=1}^{n}Q^2_n(i,j)}\to \gamma.
\end{equation}
\end{itemize}

\subsection{Graphon Convergence}\label{sec:graphon}

Below we will briefly introduce some of the basics of cut metric theory needed for our purposes, referring the audience for more details to \cite{borgs2008convergent, borgs2012convergent, lovasz2012large}.

By a \textbf{graphon}, we will mean a symmetric bounded measurable function $f:[0,1]^{2}\rightarrow[0,C_w]$. Let $\mathcal{W}$ denote the space of all graphons. Equip the space $\mathcal{W}$ by the cut distance, defined by
\[d_{\square}(f,g):=\sup_{S,T\subset [0,1]}\Big|\int_{S\times T}(f(x,y)-g(x,y))dx dy\Big|.\]
Let $\mathcal{M}$ denote the space of all measurable measure preserving maps $\iota:[0,1]\mapsto [0,1]$. Define an equivalence relation $\sim$ on the space $\mathcal{W}$ by setting
\[f\sim g\quad \text{ if }\quad  f(x,y)\stackrel{a.s.}{=}g_\iota(x,y):=g(\iota(x),\iota(y)),\text{ for some }\iota\in \cM.\]
Let $\tilde{\mathcal{W}}$ denote the quotient space $\mathcal{W}/\sim$ under the above equivalence relation. Equip $\tilde{\mathcal{W}}$ with the cut metric, defined as follows:
\[\delta_\square(\tilde{f},\tilde{g}):=\inf_{\iota\in \cM}d_\square(f,g_\iota)=\inf_{\iota\in \cM}d_\square(f_\iota,g)=\inf_{\iota_1,\iota_2\in \cM}d_\square(f_{\iota_1},g_{\iota_2}).\]
Then it follows from \cite{borgs2008convergent} that $\delta_\square$ is well defined, and  $(\tilde{\mathcal{W}},\delta_\square)$ is a compact metric space. We say a sequence of graphons $\{f_n\}_{n\ge 1}$ converge to a graphon $f$ in cut metric, if
$$\delta_\square(f_n,f)\to 0.$$
Given a symmetric $n\times n$ matrix $A_n$ with $0$ on the diagonal, define a corresponding graphon $f^{A_n}$ by setting
$$f^{A_n}(x,y):=A_{\lceil nx\rceil, \lceil ny \rceil}.$$
We say  $\{A_n\}_{n\ge 1}$ converge in cut metric to a graphon $f$, if the corresponding sequence of graphons $\{f^{A_n}\}_{n\ge 1}$ converge to $f$ in cut metric, i.e.
$$\delta_\square(f^{A_n},f)\to 0.$$

Given $f\in \mathcal{W}$, define a Hilbert-Schmidt operator $T_f$ from $L_2[0,1]$ to $L_2[0,1]$ by setting
$$T_f(g)(\cdot):=\int_{[0,1]}f(\cdot,y)g(y)dy.$$
This is a compact operator, and hence it has (at most) countably many eigenvalues. Let $\{\lambda_j\}_{j\ge 1}$ be the eigenvalues of $T_f$, arranged in decreasing order of absolute value, i.e.$$|\lambda_1|\ge |\lambda_2|\ge \cdots.$$

Throughout the paper we assume that the sequence of matrices $\{nQ_n\}_{n\ge 1}$ converge in cut metric to a graphon $f\in \mathcal{W}$ such that $\sup_{i\ge 2}\lambda_i<\lambda_1$, i.e.
\begin{align}\label{eq:cut}
\delta_\square(f^{nQ_n},f)\to 0, \quad \sup_{i\ge 2}\lambda_i<\lambda_1.
\end{align}
\textcolor{black}{Note that \eqref{RIC} and \eqref{eq:cut} together imply $\lambda_1=1$ (see \cite[Theorem 11.54]{lovasz2012large}), and so we require $\sup_{i\ge 2}\lambda_i< 1$. This last requirement essentially demands that the eigenvalue $1$ is simple, i.e.~it has multiplicity $1$.  We point out that we allow for the possibility that $|\lambda_i|=1$ for some $i$, which can only happen if $\lambda_i=-1$. And indeed, this happens for regular connected bipartite graphs (see \eqref{eq:bipartite} for an example).} We emphasize that we do not require expander type assumptions typically made in the graph theory literature, which requires $\max_{i\ge 2}|\lambda_i|<1$.   A similar spectral gap as in \eqref{eq:cut} was utilized in \cite[Equation 1.7]{deb2020fluctuations}, where the authors study universal limiting distribution of $\bar{\mathbf X}$ for Ising models on dense regular graphs. In particular, it was shown in \cite[Example 1.1]{deb2020fluctuations} that universality can fail without such an assumption. The same counter-example works in our setting as well, and demonstrates that the limiting experiment may be different without this assumption. 
\\

\begin{remark}
\textcolor{black}{
We now briefly discuss the assumptions on $Q_n$ made in this paper. Throughout we have assumed that $Q_n$ is non-negative entry-wise (and $\theta\ge 0$), which ensures that the resulting model has positive association (ferro-magnetic, in statistical physics parlance). We expect that the limiting experiment (and performance of estimators/tests) under the model \eqref{Ising} will be very different in the anti-ferromagnetic regime ($\theta<0$), or for spin glass models ($Q_n$ can take both positive and negative entries). This intuition is guided by the fact that the asymptotics of the log normalizing constant $Z_n(\theta,Q_n)$ is no longer universal under these settings (see \cite[Theorem 2.3]{basak2017universality} and  \cite[Section 1.3.2]{basak2017universality} respectively).}
\\

\textcolor{black}{The first condition \eqref{RIC} demands that the row sums of $Q_n$ are all equal. The fact that the common value of these row sums equal $1$ is just a normalization, which ensures that the \enquote{critical value} for the corresponding Ising model in \eqref{Ising} is at 1.  For non-regular matrices, we expect the limit experiment to be different, and possibly non-universal. The second condition \eqref{CWC}, which demands that all the entries of the matrix $nQ_n$ are uniformly bounded, is in essence a compactness assumption. This ensures that the sequence of functions $\{f^{nQ_n}\}_{n\ge 1}$ is uniformly bounded. Without this assumption, we again do not expect universality of the limit experiment. Assumption \eqref{CWC} also implies that assumptions \eqref{eq:frob} and \eqref{eq:cut} hold along subsequences (since $\|Q_n\|_F$ is bounded, and the space of bounded functions is compact in cut metric, as mentioned above in the graphon section). Thus \eqref{eq:frob} and \eqref{eq:cut} are in essence free assumptions, given \eqref{CWC}. In fact, convergence in cut metric (i.e.~\eqref{eq:cut}) is not needed if we only want to study the limiting experiment. The main requirement of \eqref{eq:cut} is due to the fact that we want to characterize the limit distribution of the pseudo-likelihood estimator in the critical regime.}
\\
\\
To understand the exact nature of the assumptions \eqref{RIC}, \eqref{CWC}, \eqref{eq:frob} and \eqref{eq:cut}, it is instructive to consider the commonly studied case where $Q_n$ is a scaled adjacency matrix of a graph on $n$ vertices, defined as follows:
\\
\\
Let $G_n$ be a simple labeled graph on $n$ vertices, labeled by the set $[n]$. Abusing notation slightly, we also denote by $G_n$ the adjacency matrix of the graph. Then one takes $Q_n=\frac{1}{\bar{d}}G_n$, where $\bar{d}:=\frac{1}{n}\sum_{i=1}^nd_i$ is the average degree of the graph $G_n$, and $(d_1,\ldots,d_n)$ is the labeled degree sequence of the graph $G_n$. This particular choice ensures that the resulting model is non trivial (see \cite[Corollary 1.2]{basak2017universality}). For the above choice, the assumptions \eqref{RIC} and \eqref{CWC} reduce to the following:
\begin{align*}
d_i=\sum_{j=1}^nG_n(i,j)=&\bar{d},\text{ for all }i\in [n],\quad \text{ and }\quad 
\bar{d}\ge \frac{n}{C_W}.
\end{align*}
The first assumption demands that the graph $G_n$ is regular, and the second assumption demands that the degree of $G_n$ grows linearly in $n$, i.e.~the graph $G_n$ is dense. The fourth assumption demands the convergence of the graph $G_n$ in cut metric to the function $f$. In this case, the third assumption \eqref{eq:frob} follows from \eqref{eq:cut}, as
\begin{align*}
    \sum_{i,j=1}^nQ_n^2(i,j)=\frac{1}{\bar{d}_n^2}\sum_{i,j=1}^nG_n(i,j)=\frac{n}{\bar{d}_n}\to \frac{1}{\int_{[0,1]}f(x,y)dx dy}.
\end{align*}
Thus our results apply to Ising models on dense regular graphs converging in cut metric (i.e.~under  \eqref{RIC}, \eqref{CWC} and \eqref{eq:cut}). 
\\

\end{remark}

\textcolor{black}{Below we give some concrete examples of matrix sequences $\{Q_n\}_{n\ge 1}$ which arise as scaled adjacency matrices of commonly studied graphs, for which our results apply. \begin{itemize}
\item
Suppose
\begin{align}\label{eq:cw2}
 Q_n(i,j):=\frac{1}{n}\text{ if }i\ne j.
 \end{align}
 This is the adjacency matrix of the complete graph $G_n=K_n$ on $n$ vertices, i.e.~$G_n(i,j)=1$ for all $i,j\in [n]$ such that $i\ne j$. This is a regular graph of degree $d_i=n-1$. In this case, the corresponding model in \eqref{Ising} is the well known Curie-Weiss model, given by
 \begin{align}\label{eq:cw}
\P_{\theta, {\rm CW}}({\mathbf X}={\mathbf x})=\exp\Big( \frac{n\theta\bar{x}^2}{2}-Z_n(\theta,{\rm CW})\Big).
\end{align}
 We note here that the above definition for $Q_n$ in \eqref{eq:cw2} does not exactly satisfy \eqref{RIC}, as $\sum_{j=1}^nQ_n(i,j)=\frac{n-1}{n}<1$. To satisfy \eqref{RIC}, we could have set $Q_n(i,j)=\frac{1}{n-1}$ if $i\ne j$. Nevertheless, it is slightly convenient to specify $Q_n$ using \eqref{eq:cw2}, as the eventual formulas are cleaner. Also, it is not hard to show that the Ising models obtained by the two different specifications of $Q_n$ are close in total variation, so all results carry over effortlessly between the two models.
The limiting graphon for $f_{nQ_n}$ in both these cases is the constant function $f=1$, with only one non-zero eigenvalue $\lambda_1=1$.
\\
\item
Suppose $n$ is even, and
 \begin{align}\label{eq:bipartite}
 \begin{split}
     Q_n(i,j):=&\frac{2}{n}\text{ if }1\le i\le \frac{n}{2}\text{ and }\frac{n}{2}+1\le j\le n,\\
     =&\frac{2}{n}\text{ if }1\le j\le \frac{n}{2}\text{ and }\frac{n}{2}+1\le i\le n,\\
     =&0\text{ otherwise}.
     \end{split}
     \end{align}
Thus $Q_n$ is the scaled adjacency matrix of the complete bipartite graph $G_n=K_{n/2,n/2}$, which is a regular graph of degree $d_i=\frac{n}{2}$. 
The limiting graphon for $f_{nQ_n}$ is the piecewise constant function $f$ given by
\begin{align*}
f(x,y)=&2\text{ if }0<x<1/2\text{ and }1/2<y<1,\\
=&2\text{ if }1/2<x<1\text{ and }0<y<1/2,\\
=&0\text{ otherwise}.
\end{align*}
This function has two non zero eigenvalues, $\lambda_1=1$ and $\lambda_2=-1$. In this case, we have $|\lambda_2|=|\lambda_1|$, and yet $\lambda_2<\lambda_1$ (so our results apply).
\\
\item
More generally, suppose $q\ge 2$ is a positive integer which divides $n$, and let
\begin{align*}
 \begin{split}
     Q_n(i,j):=&\frac{q}{n(q-1)}\text{ if }\Big\lceil \frac{i}{q}\Big\rceil\ne \Big\lceil \frac{j}{q}\Big\rceil \\
     =&0\text{ otherwise}.
     \end{split}
     \end{align*}
Thus $Q_n$ is the scaled adjacency matrix of the complete $q$ partite graph $G_n=K_{n/q,n/q,\cdots,n/q}$, which is a regular graph of degree $\frac{n(q-1)}{q}$. More precisely, the vertex set of $G_n$ can be partitioned into $q$ classes of equal size, such that all cross edges between classes are present. 
The limiting graphon for $f_{nQ_n}$ is the piecewise constant function $f$ given by
\begin{align*}
f(x,y)=&\frac{q}{q-1}\text{ if }\Big\lceil qx\Big\rceil \ne\Big\lceil qy\Big\rceil,\\
=&0\text{ otherwise}.
\end{align*}
This function has two non zero eigenvalues, $\lambda_1=1$ and $\lambda_2=-\frac{1}{q-1}$. Again in this case we have $\lambda_2<\lambda_1$, and so our results apply.
\\
\item
As in the example before, suppose $q\ge 3$ is a positive integer which divides $n$. But now, let
\begin{align*}\label{eq:bipartite}
 \begin{split}
     Q_n(i,j):=&\frac{q}{2n}\text{ if }\Big|\Big\lceil \frac{i}{q}\Big\rceil- \Big\lceil \frac{j}{q}\Big\rceil \Big|=1 ({\rm mod }\text{ }q),\\
     =&0\text{ otherwise}.
     \end{split}
     \end{align*}
     In words $Q_n$ is the scaled adjacency matrix of a graph $G_n$ whose vertex set can be partitioned into $q$ classes of equal size, such that all edges between consecutive classes are present, i.e.~class $2$ vertices are all connected to class $1$ vertices and class $3$ vertices, and class $q$ vertices are all connected to class $q-1$ vertices and class $1$ vertices.  Thus $G_n$ is a regular graph of degree $\frac{2n}{q}$. 
The limiting graphon for $f_{nQ_n}$ is the piecewise constant function $f$ given by
\begin{align*}
f(x,y)=&\frac{q}{2}\text{ if }\Big|\Big\lceil qx\Big\rceil -\Big\lceil qy\Big\rceil\Big|=1({\rm mod}\text{  }q),\\
=&0\text{ otherwise}.
\end{align*}
This function is has $q$ non zero eigenvalues given by $\lambda_a=\cos\Big(\frac{2a\pi}{q}\Big)$ for $0\le a\le q-1$ (using eigenvalue formulas for circulant matrices, see \cite[Chapter 3]{gray2006toeplitz}). Again in this case we have $\lambda_a<\lambda_1$ for all $a\in \{2,3,\ldots,q\}$, and so our results apply. Note that if $q$ is even, then the matrix $Q_n$ is bipartite, and the smallest eigenvalue of $f$ is $-1$. 
\\
\item
Suppose $\{d_n\}_{n\ge 1}$ is a sequence of positive integers, such that $\frac{d_n}{n}\to \eta\in (0,1)$. Let $G_n$ be a uniformly random $d_n$ regular graph, and let
\begin{align*}
Q_n(i,j)=&\frac{1}{d_n}\text{ if }(i,j) \text{ is an edge in }G_n,\\
=&0\text{ otherwise }.
\end{align*}
In this case the graph $G_n$ is $d_n$ regular by construction. It follows from \cite[Theorem 1.1]{chatterjee2011random} that $G_n$ converges in cut metric to the constant function $\eta$, and so $nQ_n$ converges in cut metric to the constant function $f\equiv 1$. Thus the behavior here is the same as that of the Curie-Weiss model (introduced in \eqref{eq:cw}), as $\lambda_1=1$ is the only non zero eigenvalue here.
\end{itemize} 
\begin{remark}
Even though we consider exactly regular graphs/matrices in this paper for the sake of clarity of arguments, we expect that our results will extend to the setting of approximately regular graphs/matrices, thereby allowing us to study Erdos-R\'enyi random graphs, stochastic block models, and Wigner matrices with non-negative entries, as done in \cite{deb2020fluctuations}.
\end{remark}
}

\subsection{Limits of Experiments}\label{sec:limits}

In this section we briefly introduce the notion of convergence of experiments. For more details we refer the reader to \cite[Chapter 9]{van2000asymptotic}.
\\

Suppose that for each $1\le n\le \infty$ we have a measure space $(\mathcal{X}_n,\mathcal{F}_n)$. Let $\{\P_{h,n},h\in H\}$ be a collection of probability measures on $(\mathcal{X}_n,\mathcal{F}_n)$. We say that $\{\P_{h,n},h\in H\}$ converges to $\{P_{h,\infty},h\in H\}$ in the sense of limits of experiments, if for every finite subset $I$ of $H$ and $h_0\in H$ we have
\begin{equation}\label{LE}
    \Big{(}\frac{d\P_{h,n}}{d\P_{h_{0},n}}(\mathbf{X})\Big{)}_{h\in I}\stackrel{d,\P_{h_{0},n}}{\longrightarrow}\Big{(}\frac{d\P_{h,\infty}}{d\P_{h_{0},\infty}}(Y)\Big{)}_{h\in I}.
\end{equation}
The RHS of \eqref{LE} is also called the likelihood ratio process with base $h_{0}$. We will use the short hand notation
$$\P_{h,n}\stackrel{\rm Exp}{\to}\P_{h,\infty}$$
to denote convergence of experiments. 

In particular, if $H\subseteq \R$ is open, and $\P_{h,\infty}=N(h,\tau^2)$ for some $\tau>0$, then we say the collection of experiments $\{\P_{h,n},h\in H\}$ is locally asymptotically normal, or LAN. 
For examples of both LAN and non LAN experiments, see \cite[Chapters 7,9]{van2000asymptotic}).

\subsection{Estimation of $\theta$}\label{sec:est}

In this paper we will focus on the behavior of two estimators for $\theta$, the maximum likelihood estimator, and the maximum pseudo-likelihood estimator.

\begin{itemize}
   \item{\bf Maximum Likelihood Estimator (MLE)}
   
The maximum likelihood estimator $\hat{\theta}_n^{MLE}$ is defined as 
$$\arg\sup_{\theta\in \R}\Big\{\frac{\theta}{2} \mathbf{X}^{T}Q_{n}\mathbf{X}- Z_n(\theta,Q_n)\Big\},$$
provided the supremum is attained uniquely.
Since the function in the above display is strictly concave in $\theta$, it follows that $\hat{\theta}_n^{MLE}$, if it exists, is the unique solution to the equation 
\begin{equation*}
   \frac{1}{2}\mathbf{X}^{T}Q_{n}\mathbf{X}=Z_n'(\theta,Q_n) =\frac{\partial Z_n(\theta,Q_n)}{\partial \theta}\Big|_{\theta=\hat{\theta}_n^{MLE}}=\frac{1}{2}\E_{\P_{\theta,Q_{n}}}\mathbf{X}^{T}Q_{n}\mathbf{X}
\end{equation*}
in $\theta\in \R$. In the special case of the Curie-Weiss model, the asymptotics of $\hat{\theta}_n^{MLE}$ was studied in  \cite{comets1991asymptotics}, where the authors demonstrated interesting phase transition properties in the limit distribution across different regimes of $\theta$. However, to the best of our knowledge, the behavior of $\hat{\theta}_n^{MLE}$ is not understood for almost any other graph sequence.

\item{\bf Maximum Pseudo-likelihood Estimator (MPLE)}

Although the MLE is a natural estimator, from a computational perspective it is often difficult to evaluate, as the normalizing constant $Z_n(\theta,Q_n)$ is computationally intractable. To bypass this, Besag introduced the maximum pseudo-likelihood estimator (\cite{besag1974spatial, besag1975statistical}) for spatial interaction models. Below we define the maximum pseudo-likelihood estimator $\hat{\theta}_n^{MPLE}$.
\\

Given  $\mathbf{X}\sim \P_{\theta,Q_{n}}$, we have
\begin{equation}
\P_{\theta,Q_{n}}(X_{i}=x_i|X_{j}\text{ for all } j\neq i):= \frac{\exp{(\theta t_ix_i)}}{\exp{(\theta t_i)}+\exp{(-\theta t_i)}},
\end{equation}
where $t_i=\sum\limits_{j=1}^nQ_n(i,j)X_{j}$. 
Define the pseudo-likelihood as the product of the above one dimensional conditional distributions:
\begin{equation*}
PL_n(\theta): = \prod\limits_{i=1}^{n}\P_{\theta, Q_n}(X_{i}|X_{j} \text{ for all } j\neq i)=\frac{\exp(\theta\sum_{i=1}^nX_it_i)}{2^n\prod_{i=1}^n \cosh(\theta t_i)}.
\end{equation*}
The maximum pseudo-likelihood estimator $\hat{\theta}_n^{MPLE}$ is defined as 
$$\arg\sup_{\theta\in \R} \log PL_n(\theta)=\arg\sup_{\theta\in \R}\{\theta \sum_{i=1}^nX_it_i-\sum_{i=1}^n\log\cosh(\theta t_i)\},$$
provided the supremum is attained uniquely. Since the function in the above display is strictly concave in $\theta$, the MPLE, if it exists, satisfies the equation
\begin{equation}\label{PL_Optimizing}
    \mathbf{X}^{T} Q_{n} \mathbf{X} = \sum\limits_{i=1}^{n}t_i\tanh{({\theta}t_i)}.
\end{equation}
The above equation \eqref{PL_Optimizing} does not involve the intractable function $Z_n(\theta,Q_n)$, and is much easier to compute. Thus computational complexity of the pseudo-likelihood estimator is much less as compared to the maximum likelihood estimator.
The consistency of the pseudo-likelihood estimator for Ising models was established in \cite{bhattacharya2018inference,chatterjee2007estimation,ghosal2020joint}, but the question of asymptotic distribution remained open.

\end{itemize}

Proposition \ref{prop_exist} gives an exact characterization for the existence of the MLE and the MPLE, and shows that the conditions hold with probability tending to $1$. 

\subsection{Hypothesis Testing for $\theta$}\label{sec:ht}

Given ${\mathbf X}\sim \P_{\theta,Q_n}$ for some $\theta>0$, suppose we want to test 
\begin{align}\label{Hypothesis_testing}
\mathcal{H}_{0}:\theta = \theta_{0}\text{\ \ \ vs\ \ \ }\mathcal{H}_{1}:\theta > \theta_n,
\end{align}
for a positive real number sequence $\{\theta_{n}\}_{n=1}^{\infty}$, at level $\alpha\in (0,1)$. Here $\theta_n$ will be chosen (depending on $\theta_0$) in such a manner, that we are in the contiguous regime, i.e.~the limiting power of the most powerful test will be between $(0,1)$. We consider three natural tests for the above problem in this paper, which are introduced below:

\begin{itemize}
\item{\bf Mean Square Test (MS-test)}

Let $\Bar{\mathbf X}$ denote the sample mean of ${\mathbf X}$, and let
\begin{align}\label{MS_test}
\psi_{n}(\mathbf{X})=\left\{
\begin{array}{rcl}
1&     &  \text{If\ \ } n\Bar{\mathbf{X}}^{2}>K_{n}(\alpha)\\
0&     &  \text{Otherwise}
\end{array} \right.
\end{align}
where $K_{n}(\alpha)$ is chosen such that $\psi_n$ has level $\alpha$. Let $\beta_{MS}$ denote the limiting power of the above test (provided the limit exists).
 
\item{\bf Neyman Pearson Test (NP-test)}
By Neyman Pearson Lemma, the UMP test for the above hypothesis testing problem is based on the sufficient statistics $\mathbf{X}^{T}Q_{n}\mathbf{X}$, and is given by
\begin{align}\label{NP_test}
\psi_{n}(\mathbf{X})=\left\{
\begin{array}{rcl}
1&     &  \text{If\ \ } \mathbf{X}^{T}Q_{n}\mathbf{X}>K_{n}(\alpha)\\
0&     &  \text{Otherwise}
\end{array} \right.
\end{align}
where $K_{n}(\alpha)$ is chosen such that the above test has level $\alpha$. It follows from standard exponential family calculations that the above test is equivalent to rejecting for large values of $\hat{\theta}_n^{MLE}$. Let $\beta_{NP}$ denote the limiting power of the above test (provided the limit exists).
 
\item{\bf Pseudo-likelihood Test (PL-test)}
With $\hat{\theta}^{MPLE}_{n}$ denoting the pseudo-likelihood estimator, define the pseudo-likelihood test by setting
\begin{align}\label{PL_test}
\psi_{n}(\mathbf{X})=\left\{
\begin{array}{rcl}
1&     &  \text{If\ \ } \hat{\theta}^{MPLE}_{n}>K_{n}(\alpha)\\
0&     &  \text{Otherwise}
\end{array} \right.
\end{align}
where $K_{n}(\alpha)$ is chosen such that the above test has level $\alpha$. Let $\beta_{PL}$ denote the limiting power of the above test (provided the limit exists).

\begin{remark}
Even though we consider a one sided testing problem in this paper, the same analysis applies to the two sided versions of the above testing problem, with obvious modifications.
\end{remark}

\end{itemize}

\subsection{Statement of main results}\label{Results}

Before we state our main results, we require some technical definitions, most of which are motivated by the following proposition from \cite[Lemma 1.1]{deb2020fluctuations}:
\begin{ppn}\label{roots_domains}
For any $\theta>0$ consider the fixed point equation
\begin{equation}\label{eq_m}
    w(\theta,x)=0, \text{\ \ \ where\ \ \ } w(\theta,x):=x-\tanh(\theta x).
\end{equation}
\begin{enumerate}
    \item[(a)] If $\theta>1$, then \eqref{eq_m} has two non-zero roots $\pm m(\theta)$ in $x$, where $m(\theta)>0$ and $\frac{\partial w(\theta,x)}{\partial x}\Big|_{x=m(\theta)}>0$. 
    \item[(b)] If $\theta=1$, then \eqref{eq_m} has a unique root $m(1)=0$, and $\frac{\partial w(\theta,x)}{\partial x}\Big|_{x=0}=0$. 
    \item[(c)] If $\theta\in (0,1)$, then \eqref{eq_m} has a unique root $m(\theta)=0$, and $\frac{\partial w(\theta,x)}{\partial x}\Big|_{x=0}>0$.  
\end{enumerate}
\end{ppn}

\begin{defn}\label{Three_domains}
Using Proposition \ref{roots_domains}, we define the function $\theta\mapsto m(\theta)$, where $m(\theta)$ is a non-negative root of the equation $w(\theta,x)=0$ \enquote{chosen carefully} as in Proposition \ref{roots_domains}. Note that $m(\theta)=0$ if $\theta\le 1$, and $m(\theta)>0$ if $\theta>1$.

We also use the above proposition to partition the parameter space $[1,\infty)$ into two distinct domains:
\begin{itemize}
    \item Low Temperature Regime: $\Theta_{1}:=(1,\infty)$;
    \item Critical point: 
    $\Theta_{2}=1$;
   \item High Temperature Regime: $\Theta_{3}:=(0,1)$.
\end{itemize}
The nomenclature of these domains is inspired from statistical physics terminology (see for e.g.~\cite{ellis1978statistics}). \textcolor{black}{It follows from \cite[Theorem 3.3 part (a)]{bhattacharya2018inference} that consistent estimation/testing for $\theta$ is not impossible in the model $\P_{\theta,Q_n}$ for $\theta\in (0,1)$. This is also indicated by the fact that for $\theta\in (0,1)$ the function $m(.)$ is constant. Thus the asymptotics in this regime is of less statistical interest, and we focus only on the regime $\theta\ge 1$ for the rest of this paper. For the interested reader, we refer to version 1 of this article which also treats the regime $\theta\in (0,1)$ (\cite{xu2022ising}).}

For any $\theta\in \Theta_1 $, define the positive real $\sigma^2(\theta)$  by setting 
\begin{align}\label{eq:sigma}
\sigma^2(\theta):=\frac{1-m^2(\theta)}{1-\theta(1-m^2(\theta))}.
\end{align}
Note that $\sigma^2(\theta)$ is well defined by Proposition \ref{roots_domains}, as $1-(1-m^2(\theta))\theta=\frac{\partial w(\theta,x)}{\partial x}\Big|_{x=m(\theta)}>0$. 
\end{defn}

\begin{defn}\label{def:thetan}
Given $h\in \R$ and $\theta_0\ge 1$, define a positive sequence $\{\theta_n\}_{n\ge 1}$ (depending on $h,\theta_0$) by setting $\theta_n:=\theta_0+\frac{h}{\sqrt{n}}$.
We will omit the dependence of $\theta_0,h$, since it will be clear from the context.
\end{defn}

\begin{defn}
Given any continuous real valued random variable $\zeta$, let $\Psi_\zeta(.):(0,1)\mapsto \R$ denote the quantile function, i.e.~the inverse cdf of $\zeta$, defined by
$$\Psi_\zeta(p):=\inf\{t\in \R:F_\zeta(t)\ge p\},$$
where $F_\zeta$ is the cdf of $\zeta$. In particular if $\zeta\sim N(0,1)$, then we will also use the notation $z_\alpha$ for $\Psi_\zeta(1-\alpha)$, as is standard in statistics literature.
\end{defn}

\begin{defn}\label{def:kappa}
Let $\kappa:=\gamma^2-\|f\|_2^2$, where $\gamma,f$ are as in \eqref{eq:frob} and \eqref{eq:cut} respectively.
\end{defn}

We now state the results for each of the domains separately.

\subsubsection{The low temperature regime $\Theta_1$}

Our first theorem describes the limits of experiments, asymptotic performance of estimators, and asymptotic performance of tests in low temperature regime $\Theta_{1}$. 
\begin{thm}\label{test_low}
Suppose ${\mathbf X}\sim \P_{\theta,Q_n}$
with $Q_{n}$ satisfying \eqref{RIC}, \eqref{CWC}, \eqref{eq:frob} and \eqref{eq:cut}  for some $C_W, \kappa\in (0,\infty)$ and $f\in \mathcal{W}$.
Then with $R(\theta_0):=m^2(\theta_0)\sigma^2(\theta_0)$, the following conclusions hold:
\begin{enumerate}
\item[(a)]\label{low_a} We have 
\begin{equation*}
        \{\P_{\theta_{0}+h n^{-1/2},Q_{n}}\}_{h\in \R}\stackrel{\rm Exp}{\longrightarrow}\{N\Big(h, R(\theta_0)^{-1}\Big)\}_{h\in \R}.
    \end{equation*}

\item[(b)]\label{low_est} The MLE $\hat{\theta}_{n}^{MLE}$ and the MPLE $\hat{\theta}_{n}^{MPLE}$ (as defined in section \ref{sec:est}) exist with probability tending to $1$, and have a common asymptotic distribution, given by
\begin{equation*}
\sqrt{n}(\Tilde{\theta}_{n}^{MLE}-\theta_{0})\stackrel{d}{\to}N\Big(0,R(\theta_0)^{-1}\Big),\quad
\sqrt{n}(\hat{\theta}^{MPLE}_{n}-\theta_{0})\stackrel{d}{\to}N\Big(0,R(\theta_0)^{-1}\Big).
\end{equation*}

\item[(c)]\label{low_b} 
With  $\theta_n=\theta_0+\frac{h}{\sqrt{n}}$ (as in definition \ref{def:thetan}) for some $h>0$, and $\beta_{MS}$, $\beta_{NP}$ and  $\beta_{PL}$ as defined in section \ref{sec:ht},  we have
\begin{equation*}
    \beta_{NP} = \beta_{PL} = \beta_{MS}= \P(N(0,1)>z_\alpha-h\sqrt{R(\theta_0)}),
\end{equation*}
where 
$z_\alpha$ represents the $(1-\alpha)^{th}$ quantile for $N(0, 1)$. \end{enumerate}
\end{thm}
\begin{remark}
The above theorem shows that for $\theta_0>1$, the family of Ising models is LAN at scale $n^{-1/2}$, which is what happens in classical statistics for iid models. It then follows by extension of classical arguments that $\hat{\theta}_n^{MLE}$ is asymptotically optimal. Perhaps surprisingly, part (b) above shows that $\hat{\theta}_n^{MPLE}$ (which requires significantly less computational resources) is also asymptotically optimal. Carrying this through, it is shown in part (c) that the tests based on $\hat{\theta}_n^{MLE}$ and $\hat{\theta}_n^{MPLE}$ have the same asymptotic power. In fact, the much simpler test based on the sample mean $\bar{X}$ also has the same asymptotic power, computation of which does not even require the knowledge of the matrix $Q_n$. Thus in this regime it is possible to gain optimal asymptotic performance for tests of hypothesis without the knowledge of $Q_n$. 
\end{remark}

\subsubsection{The critical regime $\Theta_2$}
As demonstrated in our next result, the behavior is very different when $\theta_0=1$ (which is the critical point). To describe the limit experiment (which is no longer LAN), and the limiting behavior of estimators/tests we make the following definitions.
\begin{defn}\label{def:H_h}
Let $\{\mathbb{H}_{h}(\cdot),h\in \R\}$ be a family of probability distributions on $\R$ parametrized by $h$, with density function 
\begin{equation}\label{H_h}
    p_{h}(u)= \exp{(-\frac{1}{12}u^{4}+\frac{1}{2}hu^{2}-F(h))}.
\end{equation} 
Here $$F(h):=\log \int_{\R} \exp{(-\frac{1}{12}u^{4}+\frac{1}{2}hu^{2})}du$$ is the log normalizing constant, which makes $p_h(.)$ into a density. Let $U_{h}\sim \H_h$.
\\

\textcolor{black}{Let $W^*\sim N(0,2\kappa)$ and $\{Y_j\}_{j\ge 2}\stackrel{iid}{\sim}\chi_1^2$ be  mutually independent. For $\theta\ge 0$, define two random variables $S_\theta$ and $T_\theta$ by setting
\begin{align}\label{eq:ss}
S_\theta:=&(1-m^2(\theta))\Big[\sum\limits_{j=2}^{\infty}\lambda_{j}\Big(\frac{Y_{j}}{1-\theta(1-m^2(\theta))\lambda_{j}}-1\Big)-1+(1-m^2(\theta))\theta\kappa+W^*\Big],\\
\label{eq:tt}
T_\theta:=&(1-m^2(\theta))\Big[\sum\limits_{j=2}^{\infty}\frac{\lambda_{j}^{2}Y_{j}}{1-\theta(1-m^2(\theta))\lambda_{j}}+\kappa\Big],
\end{align}
where $m(\theta)$ is as in definition \ref{Three_domains}. 
Here the infinite sums above converge in $L_2$ (see Lemma \ref{lem:comparison} in the supplementary file). 
 Set
\begin{equation}\label{eq:V_h}
       V_{h}=\frac{1}{3}U_{h}^{2}+\frac{S_1-T_1}{U_{h}^{2}},
   \end{equation}
   where $U_h$ is independent of $(S_1,T_1)$, and $(S_1,T_1)$ is as defined in \eqref{eq:ss} and \eqref{eq:tt}).}
\end{defn}

\begin{thm}\label{test_critical}
Suppose ${\mathbf X}\sim \P_{\theta,Q_n}$
with $Q_{n}$ satisfying \eqref{RIC}, \eqref{CWC}, \eqref{eq:frob} and \eqref{eq:cut}  for some $C_W, \kappa\in (0,\infty)$ and $f\in \mathcal{W}$.
Then the following conclusions hold:
\begin{enumerate}
    \item[(a)]\label{critical_a}We have
    \begin{equation}\label{critical_limit}
    \{\P_{\theta_{0}+hn^{-1/2},Q_{n}}\}_{h\in \R}\stackrel{\rm Exp}{\longrightarrow}\{\mathbb{H}_{h}\}_{h\in \R}.
    \end{equation}
    
    \item[(b)]\label{critical_est} The MLE $\hat{\theta}_n^{MLE}$ and MPLE
$\hat{\theta}_{n}^{MPLE}$ as defined in section \ref{sec:est} exist with probability tending to $1$, and satisfy
\begin{align}\label{MLE_critical}
\P_{1,Q_{n}}\big(\sqrt{n}(\hat{\theta}_{n}^{MLE}-1)\leq h\big)\longrightarrow &\P\big(U_{0}^{2}\leq\E U_{h}^{2} \big),\\
\label{MPLE_critical}
\sqrt{n}(\hat{\theta}^{MPLE}_{n}-1)\stackrel{d,\P_{1,Q_{n}}}{\longrightarrow}&V_{0},
\end{align}
    
\item[(c)]\label{critical_b} 
With  $\theta_n=1+\frac{h}{\sqrt{n}}$ for some $h>0$ (as in definition \ref{def:thetan}), and $\beta_{MS}$, $\beta_{NP}$ and  $\beta_{PL}$ as defined in section \ref{sec:ht},  we have
    \begin{align}
    \beta_{NP} = \beta_{MS}=&2\P\big(U_{h}>\Psi_{\mathbb{H}_0}(1-\alpha/2)\big),\\
    \beta_{PL}=&\P\big(V_{h}>\Psi_{V_{0}}(1-\alpha)\big).
    \end{align} 
\end{enumerate}
\end{thm}

\begin{remark}
Thus, similar to Theorem \ref{test_low}, the contiguous alternatives obtained in Theorem \ref{test_critical} in the critical regime is also of size $O(\frac{1}{\sqrt{n}})$. However, the limit experiment is no longer gaussian, and so we are outside the familiar LAN setting of classical settings. This is also reflected through the non-gaussian limiting distributions for MLE and MPLE. In terms of tests of hypothesis, part (c) shows that there is a discrepancy between the asymptotic powers of PL-Test and NP-test. The MS-Test continues to be asymptotically optimal, and thus provides a computationally efficient $Q_n$ agnostic solution to our testing problem. We show in section \ref{sec:example} below that $\beta_{MS}=\beta_{PL}$ for the Curie-Weiss model.
\end{remark}

\subsection{Examples}\label{sec:example}
To illustrate our results, we will now apply our main theorems and simulation results on two concrete examples, the Curie-Weiss model  \eqref{eq:cw}, and the Ising model on the complete bi-partite graph $K_{n/2,n/2}$ (defined in \eqref{eq:bipartite}). Table \ref{tab:1} compares the asymptotic distribution of the MLE and the MPLE under both the Curie-Weiss Model and the bipartite Ising model. In the low temperature regime $\Theta_1$, the asymptotic  distribution of both the MLE and MPLE is $N(0,R(\theta_0)^{-1})$, for both the Curie-Weiss model and the bipartite Ising model (see Theorem \ref{test_low} part (b)). In fact, the same universal limit continues to hold for any sequence of graphs $G_n$ satisfying \eqref{RIC}, \eqref{CWC}, \eqref{eq:frob} and \eqref{eq:cut}. 
At the critical regime $\Theta_2$, the asymptotic distribution of the two estimators are not the same, for either the Curie-Weiss model (see figure \ref{fig:a}, or the bipartite Ising model (see figure \ref{fig:b}). The simulations in figures \ref{fig:a} and \ref{fig:b} use the limiting distribution  obtained in Theorem \ref{test_critical} part (b). \textcolor{black}{For the sake of completeness, we note that in the high temperature regime all estimators (and hence MLE and MPLE) are both inconsistent  (\cite[Theorem 3.3 part (a)]{bhattacharya2018inference}).}
\\

\begin{center}\label{tab:1}
\begin{tabular}{ |p{4cm}||p{5cm}|p{5cm}| }
 \hline
 \multicolumn{3}{|c|}{Table 1: Asymptotic distribution of MLE vs MPLE} \\
 \hline
 Regimes of $\theta$ & Curie-Weiss Model & Bipartite Ising Model\\
 \hline
 Low Temperature   &   normal, same limit   &  normal, same limit \\
 Critical Point&  non-normal, \textcolor{blue}{different limit}   &  non-normal, \textcolor{blue}{different limit}  \\
 \hline
\end{tabular}
\end{center}

In a similar manner, table \ref{tab:2} compares the  asymptotic powers of the three tests (NP, MS and PL) for the Curie-Weiss Model and the bipartite Ising model. It follows from the main theorems that the three tests have the same asymptotic power if either $\theta\in \Theta_1$ or we are in the Curie-Weiss setting. At criticality the NP test and MS test have the same power, but the PL test has a lower power for the bipartite Ising model  (see figure \ref{fig:td}). 

\begin{center}\label{tab:2}
\begin{tabular}{ |p{4cm}||p{4cm}|p{4cm}|p{4cm}|  }
 \hline
 \multicolumn{3}{|c|}{Table 2: Aymptotics powers of tests} \\
 \hline
 Regimes of $\theta$ & Curie-Weiss Model & Bipartite Ising Model\\
 \hline
 Low Temperature   &  $\beta_{NP}=\beta_{MS}=\beta_{PL}$  & $\beta_{NP}=\beta_{MS}=\beta_{PL}$\\
 Critical Point& $\beta_{NP}=\beta_{MS}=\beta_{PL}$   & \textcolor{blue}{$\beta_{NP}=\beta_{MS}>\beta_{PL}$}  \\
 \hline
\end{tabular}
\end{center}

\begin{figure}[t!] 
\begin{subfigure}{0.49\textwidth}
\includegraphics[width=\linewidth]{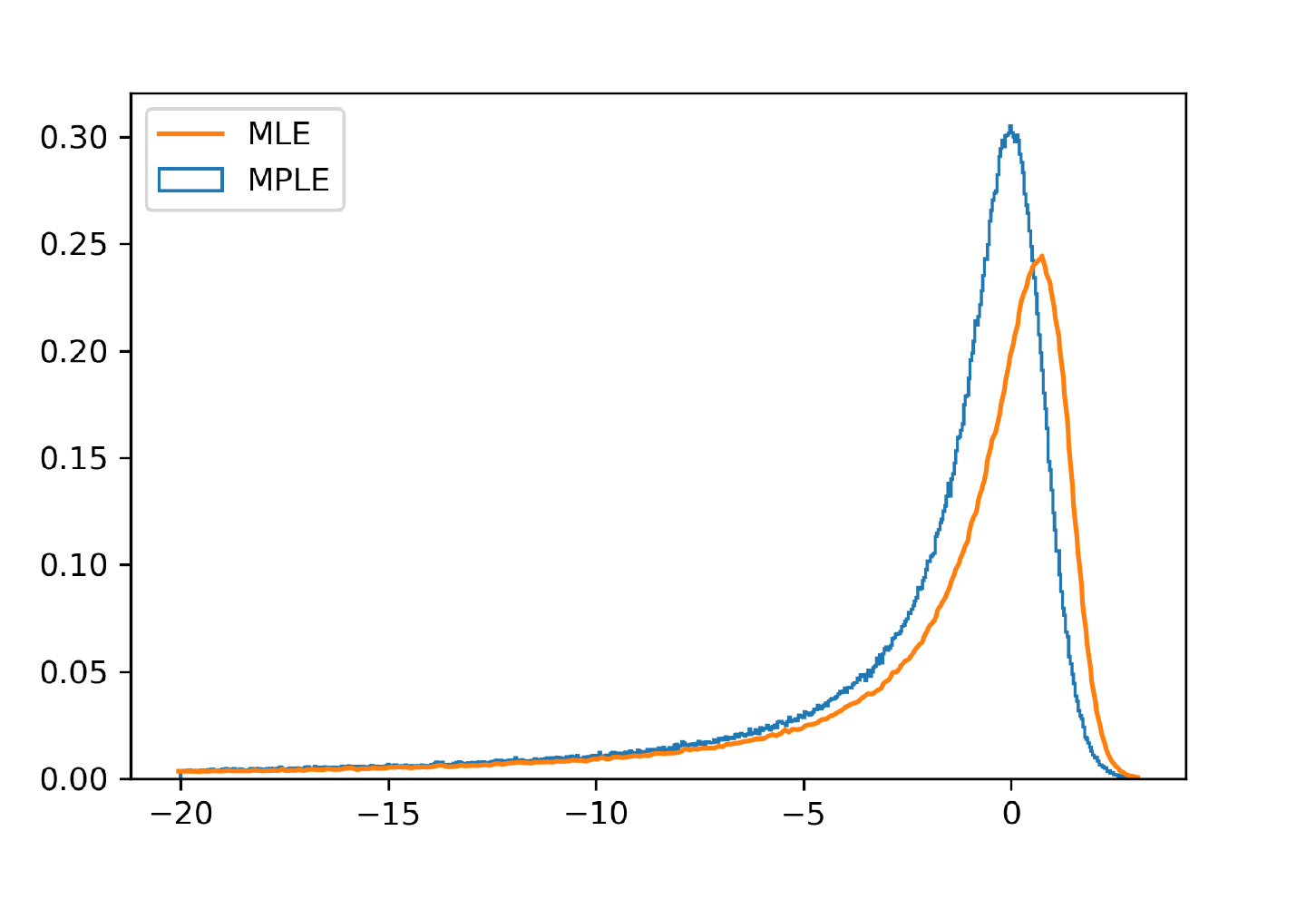}
\caption{Curie Weiss Model, $\theta_{0}=1$} \label{fig:a}
\end{subfigure}
\begin{subfigure}{0.49\textwidth}
\includegraphics[width=\linewidth]{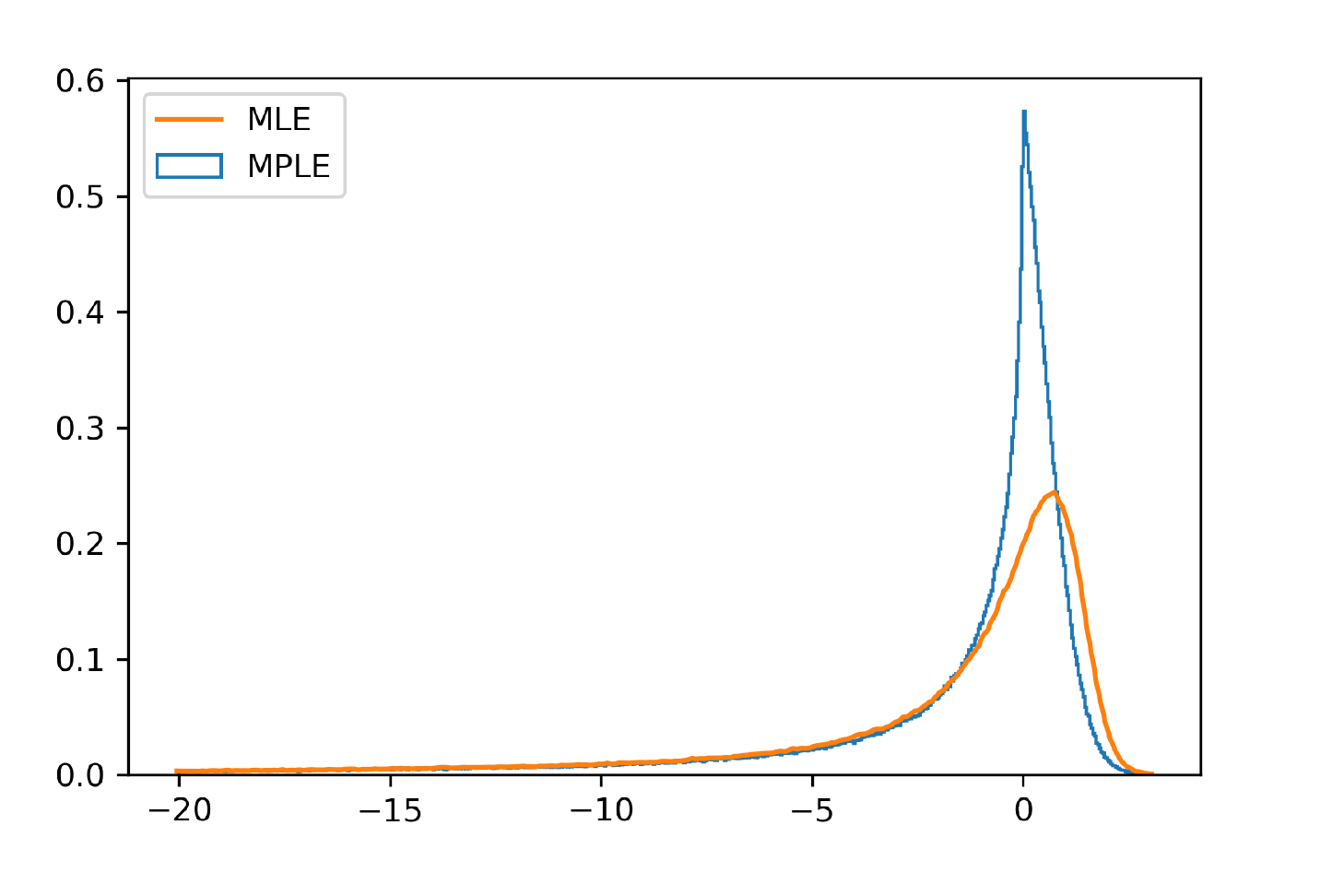}
\caption{Bipartite Ising Model, $\theta_{0}=1$} \label{fig:b}
\end{subfigure}
\caption{\textit{The left panel shows the pdf of limiting distributions of MPLE and MLE under Biparite Ising model at critical point, and the right panel shows that under Curie-weiss model at critical point.  
}} \label{fig:0}
\end{figure}

\begin{figure}[t!] 
\begin{subfigure}{0.49\textwidth}
\includegraphics[width=\linewidth, height=6cm]{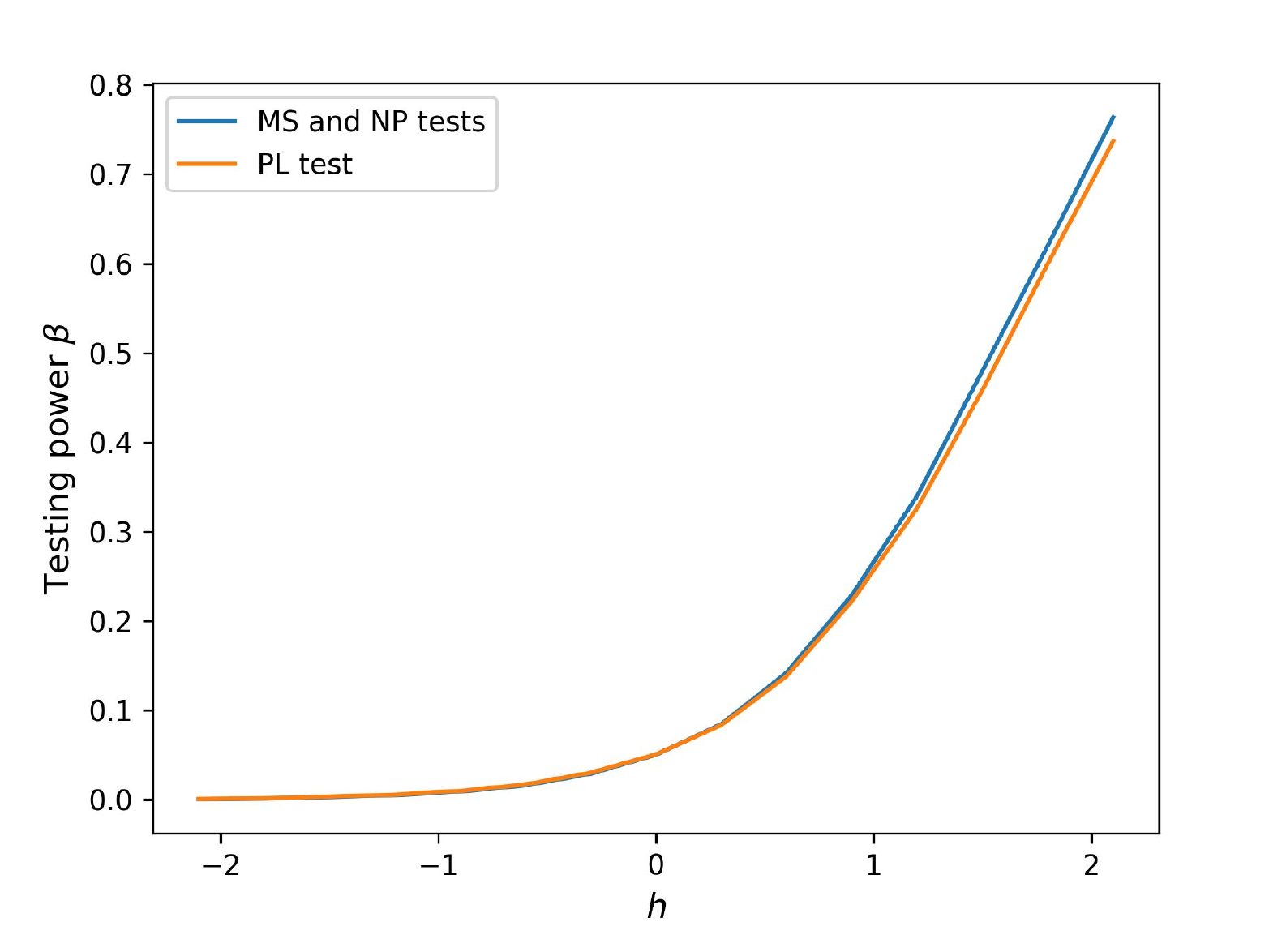}
\caption{Bipartite Ising Model, $\theta_{0}=1$} \label{fig:td}
\end{subfigure}
\begin{subfigure}{0.49\textwidth}
\includegraphics[width=\linewidth, height=6cm]{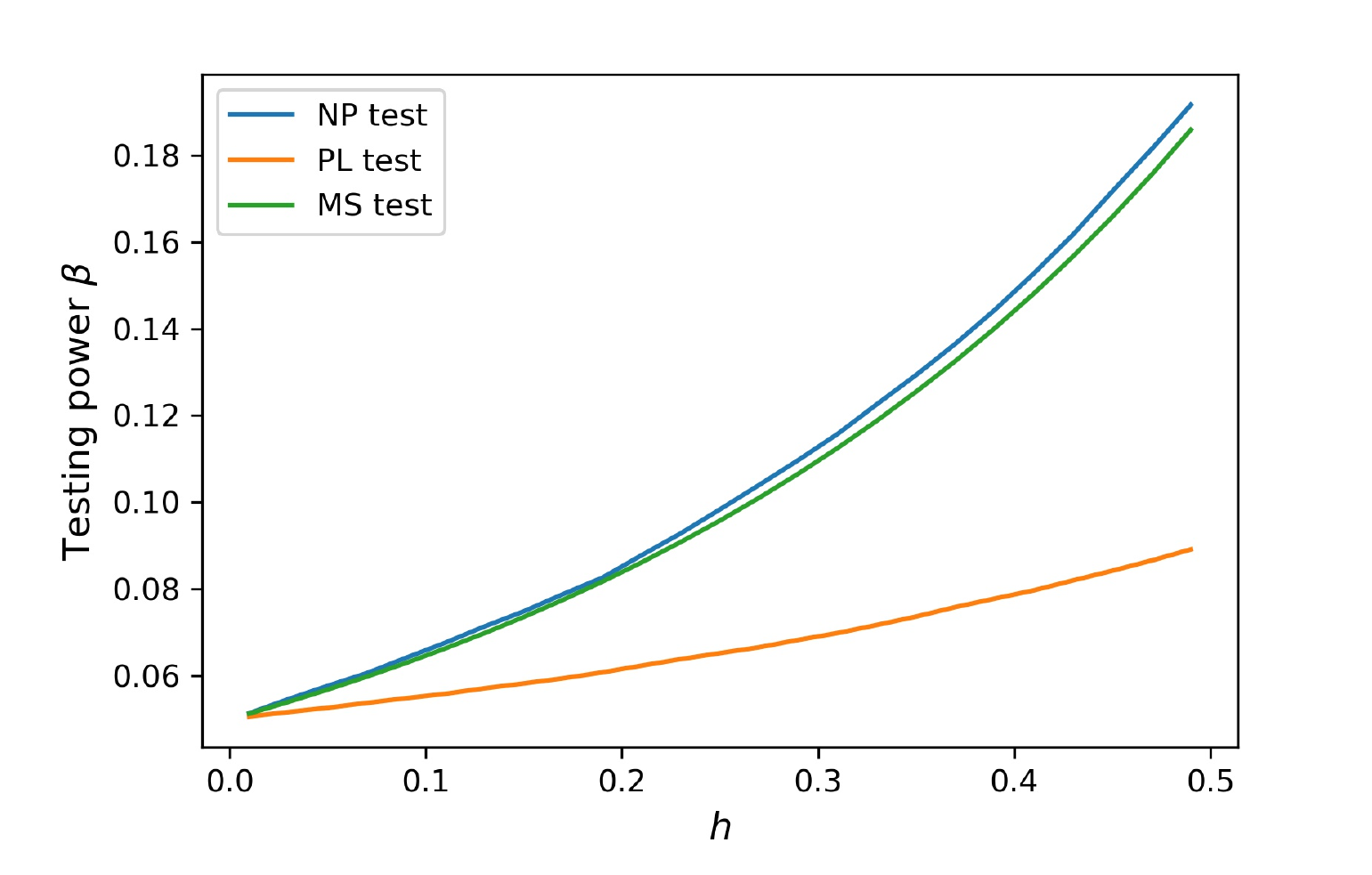}
\caption{Bipartite Ising Model, $\theta_{0}=0.1$} \label{fig:tf}
\end{subfigure}

\caption{\textit{Both panels show the testing powers versus $h$ under biparite Ising model, at the critical regime and the high temperature regime, respectively. 
}} \label{fig:1}
\end{figure}

\subsection{Main Contributions \& Future Scopes}\label{sec:222}
In this paper we establish limits of experiments for a class of one parameter Ising models on dense regular matrices. We show that the limiting experiment is universal (i.e.~does not depend on the graph sequence) and LAN in the low temperature regime, and is universal and non LAN in the critical regime. 
Using the tools developed, we analyze the performance of commonly studied estimators and tests of hypothesis, and compare their performance across different regimes. One surprising discovery is that the asymptotic performance of the MLE and the MPLE is the same in the low temperature regime, thus demonstrating that the extra computational burden of the MLE has no asymptotic gain in terms of statistical accuracy. In terms of tests of hypothesis, there is a more computationally efficient test (compared to tests based on either MLE or MPLE) based on the sample mean, which matches the optimal power function in low and critical regimes. Prior to this work, such detailed inferential result was largely non existent. We demonstrate our results by applying them to Ising models on 
(i) complete graph, and
(ii) complete bi-partite graph.
\\

Throughout this paper we focus on Ising models on dense regular matrices. It would be interesting to develop inference for Ising models on matrices which are either non regular, or non dense. Another direction of future interest is to consider the case when the Ising model has a non-zero magnetic field, and study joint inference for both the inverse temperature parameter, and the magnetic field. A more challenging problem is to extend these results to cubic and other higher order interaction models, similar to the Exponential Random Graph Models of social sciences.

\subsection{Outline}
 The rest of this paper is organized as follows:

\textcolor{black}{ In section \ref{sec:main_results} below, we give the proofs of our the main theorems. The proofs of the theorems in section \ref{sec:main_results} rely on two key lemmas about the Ising model (Lemma \ref{lem:mean} and Lemma \ref{lem:normalizing_ising}),  the proofs of which are deferred to the appendix (section \ref{sec:appen}). }

\section{Proofs of Main Theorems}\label{sec:main_results}

We begin by stating the following lemmas, which we will use to prove our main results. 

Our first lemma computes the limiting distributions of various quantities under the Ising model.

\begin{lem}\label{lem:mean}
Let ${\mathbf X}\sim \P_{\theta_n,Q_n}$, where $\theta_0>0,h\in \R$, and $\theta_n$ is as in definition \ref{def:thetan}. Assume that the matrix $Q_n$ satisfies \eqref{RIC}, \eqref{CWC}, \eqref{eq:frob} and \eqref{eq:cut}  for some $C_W, \kappa\in (0,\infty)$ and $f\in \mathcal{W}$. Also let $S_{\theta_0},T_{\theta_0}$ be random variables as defined in \eqref{eq:ss} and \eqref{eq:tt} respectively. Then, setting  $B_n:=Q_n-\frac{1}{n}{\bf 1}{\bf 1}^T$, the following conclusions hold:

\begin{enumerate}
    \item[(a)]
    If $\theta_0\in \Theta_1$, then conditional on the set $\bar{\mathbf X}>0$ we have
    $$\Big[\sqrt{n}(\bar{\mathbf X}-m(\theta_n)), {\mathbf X}^TB_n{\mathbf X}, {\mathbf X}^TB_n^2{\mathbf X}\Big]\stackrel{d}{\to}[W_{\theta_0},S_{\theta_0},T_{\theta_0}].$$
    where $m(.)$ is as in definition \ref{Three_domains}, and $W_{\theta_0}\sim N(0,\sigma^2(\theta_0))$ is independent of $(S_{\theta_0},T_{\theta_0})$.
    
    \item[(b)]
    If $\theta_0\in \Theta_2$, then 
    $$\Big[n^{1/4}\bar{\mathbf X},{\mathbf X}^TB_n{\mathbf X}, {\mathbf X}^TB_n^2{\mathbf X}\Big]\stackrel{d}{\to}[U_h,S_{\theta_0},T_{\theta_0}],$$
    where $U_h\sim\mathbb{H}_h$ (see \eqref{H_h}) is independent of $(S_{\theta_0},T_{\theta_0})$.
    
%
\end{enumerate}
\end{lem}
Our second lemma gives very precise asymptotics for the normalizing constant of Ising models.
\begin{lem}\label{lem:normalizing_ising}
Let $\theta_0>0,h\in \R$, and let $\theta_n$ be as in definition \ref{def:thetan}. Assume that the matrix $Q_n$ satisfies \eqref{RIC}, \eqref{CWC}, \eqref{eq:frob} and \eqref{eq:cut}  for some $C_W, \kappa\in (0,\infty)$ and $f\in \mathcal{W}$.
\begin{enumerate}
    \item[(a)] If $\theta_0>1$ then we have
    $$\lim_{n\to\infty}\Big\{Z_n(\theta_n,Q_n)-Z_n(\theta_0,Q_n)-\frac{1}{2}\sqrt{n}hm^2(\theta_0)\Big\}=\frac{R(\theta_0)h^2}{2}. $$
    
    \item[(b)]
    If $\theta_0=1$, then we have
$$\lim_{n\to\infty}\{Z_n(\theta_n,Q_n)-Z_n(\theta_0,Q_n)\}=F(h)-F(0),$$
where $F(h)$ is as defined in \eqref{H_h}.

%
\end{enumerate}
\end{lem}
The proofs of Lemma \ref{lem:mean} and Lemma \ref{lem:normalizing_ising} are given in the supplementary file.
\\

Our third result gives an exact characterization of the existence of MLE and MPLE, and show that they exist with high probability.
\begin{ppn}\label{prop_exist}
\begin{enumerate}
    \item[(a)]
    Let $$a_n:=\min_{{\bf x}\in \{-1,1\}^n}{\bf x}^T Q_n{\bf x}\le \max_{{\bf x}\in \{-1,1\}^n}{\bf x}^T Q_n{\bf x}=:b_n.$$
    Then the MLE exists in $\R$ iff
    $$a_n<{\mathbf X}^T Q_n{\mathbf X}<b_n.$$
    
    \item[(b)]
    
    In particular, the MLE exists with probability tending to 1 for all $\theta_0\in \Theta_1\cup \Theta_2$.

    \item[(c)]
    Let $$S=S({\bf X}):=\{i:t_i\ne 0\}.$$ Then the MPLE exists in $\R$ iff neither of the following two events happen:
    
    \begin{itemize}
        \item 
        ${\mathbf X}_i=1\text{ for all }i\in S$.
        
        \item
          ${\mathbf X}_i=-1\text{ for all }i\in S$.
    \end{itemize}
    
    \item[(d)]
    
    In particular, the MPLE exists with probability tending to $1$ for all $\theta_0\in \Theta_1\cup \Theta_2$.

\end{enumerate}
\end{ppn}

The proof of Proposition \ref{prop_exist} is given after the proof of the main results.
\\

Our final result is a calculus proposition which computes derivatives of the function $m(.)$ defined in \ref{Three_domains}.

\begin{ppn}\label{ppn:m}

Let $m(.):(1,\infty)\mapsto (0,1)$ be as in definition \ref{Three_domains}. Then with $\theta_n=\theta_0+\frac{h}{\sqrt{n}}$ (as in definition \ref{def:thetan}) for some $\theta_0\in \Theta_1$ and $h>0$ we have
$$\lim_{n\to\infty}\frac{m(\theta_n)-m(\theta_0)}{\frac{h}{\sqrt{n}}}=m(\theta_0)\sigma^2(\theta_0).$$
\end{ppn}

The proof of Proposition \ref{ppn:m} is deferred to the supplementary file.


\subsection{Proof of Theorem \ref{test_low}}

Before we begin, we use Lemma \ref{lem:mean} parts (a) and (b) to conclude that if  $(\theta_0,h)\in (\Theta_1\cup \Theta_2,\R)$ we have ${\mathbf X}^TB_n{\mathbf X}=O_p(1)$ under $\P_{\theta_0+hn^{-1/2},Q_n}$, and so
\begin{align}\label{eq:all_regimes}
  {\mathbf X}^T Q_n{\mathbf X}=n\bar{\mathbf X}^2+{\mathbf X}^TB_n{\mathbf X}=n\bar{\mathbf X}^2+O_p(1).
\end{align}
$ $\newline
\textbf{\textit{Part (a):}}

Let $I=\{h_1,\ldots,h_k\}$ with $h_1<h_2<\ldots<h_k$ for some positive integer $k$. Then setting $\theta_{n,i}:=\theta_0+h_i n^{-1/2}$  we have
\begin{small}
\begin{align}\label{eq:log1}
\log \frac{d\P_{\theta_{n,i},Q_{n}}(\mathbf{X})}{d\P_{\theta_{0},Q_{n}}(\mathbf{X})}
= Z_n(\theta_0,Q_n)- Z_n(\theta_{n,i},Q_n)+\frac{h_i\sqrt{n}m^2(\theta_0)}{2} +\frac{h_i}{2}\Big(\frac{{\mathbf X}^T Q_n{\mathbf X}}{\sqrt{n}}-\sqrt{n}m^2(\theta_0)\Big).
\end{align}
\end{small}
By part (a) of Lemma \ref{lem:normalizing_ising} we have 
\begin{align*}
    Z_n(\theta_{n,i},Q_n)-Z_n(\theta_0,Q_n)-\frac{h_i\sqrt{n}m^2(\theta_0)}{2}
    \to \frac{R(\theta_0)h_i^2}{2}.
\end{align*}
Also, using \eqref{eq:all_regimes}, under $\P_{\theta_0,Q_n}$ we have
\begin{align}\label{eq:claim2}
\frac{1}{\sqrt{n}}{\mathbf X}^T Q_n{\mathbf X}-\sqrt{n}m^2(\theta_0)=\sqrt{n}\Big(\bar{\mathbf X}^2-m^2(\theta_0)\Big)+O_p\Big(\frac{1}{\sqrt{n}}\Big)\stackrel{d}{\to}2m(\theta_0)W_{\theta_0},
\end{align}
where the last step we use part (a) of Lemma \ref{lem:mean} to conclude that under $\P_{\theta_0,Q_n}$ we have
    \begin{align}\label{eq:partb}\sqrt{n}\Big(\bar{\mathbf X}^2-m^2(\theta_0)\Big)=\sqrt{n}\Big(\bar{\mathbf X}-m(\theta_0)\Big) \Big(\bar{\mathbf X}+m(\theta_0)\Big)\stackrel{d}{\to}2m(\theta_0)W_{\theta_0},    \end{align}
where $W_{\theta_0}\sim N(0,\sigma^2(\theta_0))$.
The above calculation gives that under $\P_{\theta_0,Q_n}$ we have
\begin{align*}
   \Big[ \log \frac{d\P_{\theta_{n,i},Q_{n}}(\mathbf{X})}{d\P_{\theta_{0},Q_{n}}(\mathbf{X})}\Big]_{1\le i\le k}
   \stackrel{d}{\to}&\Big[-\frac{R(\theta_0)h_i^2}{2}+m(\theta_0)h_i W_{\theta_0}\Big]_{1\le i\le k}.
\end{align*}
Also note that if $W\sim N(0,R(\theta_0)^{-1})$ then
\begin{align*}\Big[\log \frac{dN(h_i,R(\theta_0)^{-1})}{dN(0, R(\theta_0)^{-1})}(W)\Big]_{1\le i\le k}=&\Big[- \frac{ R(\theta_0)h_i^2}{2}+R(\theta_0)Wh_i\Big]_{1\le i\le k}\\
\stackrel{d}{=}&\Big[-\frac{R(\theta_0)h_i^2}{2}+m(\theta_0)h_i W_{\theta_0}\Big]_{1\le i\le k}.
\end{align*}
It then follows from the last two displays that under $\P_{\theta_0,Q_n}$ we have
 $$\Big[ \log \frac{d\P_{\theta_{n,i},Q_{n}}(\mathbf{X})}{d\P_{\theta_{0},Q_{n}}(\mathbf{X})}\Big]_{1\le i\le k}
   \stackrel{d}{\to}\Big[\log \frac{dN(h_i,R(\theta_0)^{-1})}{dN(0, R(\theta_0)^{-1})}(W)\Big]_{1\le i\le k},$$
which verifies convergence of experiments.

\textbf{\textit{Part (b):}}

\begin{itemize}
\item{\bf MLE}

Suppose ${\mathbf X}\sim \P_{\theta,Q_n}$. 
Using Proposition \ref{prop_exist} part (b), it follows that the MLE exists with probability tending to $1$. 
Also, the MLE $\hat{\theta}_n^{MLE}$ satisfies the equation $$Z_n'(\hat{\theta}_n^{MLE},Q_n)=\frac{1}{2}{\mathbf X}^T Q_n{\mathbf X}.$$
Proceeding to find the limit distribution of $\hat{\theta}^{MLE}_n$, for any $h\in \R$ setting $\theta_n:=\theta_0+h/\sqrt{n}$ we have
\begin{align}\label{eq:mle_low}
\notag&\P_{\theta_{0},Q_{n}}\big(\sqrt{n}(\hat{\theta}_{n}^{MLE}-\theta_{0})\leq h\big)\\
\notag=&
\P_{\theta_{0},Q_{n}}\Big({\mathbf X}^T Q_n{\mathbf X}\le 2 Z_n'(\theta_n)\Big)\\
=&\P_{\theta_{0},Q_{n}}\Big(\frac{1}{\sqrt{n}}{\mathbf X}^T Q_n{\mathbf X}-\sqrt{n}m^2(\theta_{0})\le \frac{2Z_n'(\theta_n,Q_n)}{\sqrt{n}}-\sqrt{n}m^2(\theta_{0})\Big).
\end{align}
We now claim that
\begin{align}\label{eq:claim3}
\frac{2Z_n'(\theta_n,Q_n)}{\sqrt{n}}-\sqrt{n}m^2(\theta_0)=2hR(\theta_0)+o(1).
\end{align}
Given \eqref{eq:claim3}, using \eqref{eq:claim2} and \eqref{eq:mle_low}  we get
\begin{align*}
    \P_{\theta_{0},Q_{n}}\big(\sqrt{n}(\hat{\theta}_{n}^{MLE}-\theta_{0})\le h\big)=
& =\P_{\theta_{0},Q_{n}}\big(\frac{1}{\sqrt{n}}\mathbf{X}^{T}Q_{n}\mathbf{X}-\sqrt{n}m^{2}(\theta_0)\leq2hR(\theta_0)+o(1)\big)\\
&\longrightarrow \P\Big(N\Big(0,R(\theta_0)^{-1}\Big)\leq h\Big),
\end{align*}
which verifies asymptotic distribution for the MLE. 
\\

Proceeding to verify \eqref{eq:claim3}, note that the LHS of the display in part (a) of Lemma \ref{lem:normalizing_ising} is convex in $h$, and converges point-wise to a convex function which is differentiable everywhere. It follows that the derivatives of the functions also converge, which gives \eqref{eq:claim3}.

\item{\bf MPLE}

It follows from \cite[Corollary 3.1 (b)]{bhattacharya2018inference} that if $\theta_0>1$, the MPLE $\hat{\theta}_n^{MPLE}$ exists with probability going to $1$, and is $\sqrt{n}$ consistent (existence without consistency also follows from Proposition \ref{prop_exist}). It thus suffices to prove asymptotic normality. To this effect, we show the more general result that for any $h\in \R$, setting $\theta_n:=\theta_0+\frac{h}{\sqrt{n}}$, under $\P_{\theta_n,Q_n}$ we have
\begin{align}\label{eq:pl_gen}
\sqrt{n}(\hat{\theta}_n^{MPLE}-\theta_0)\stackrel{d}{\to}N\Big(h, R(\theta_0)^{-1}\Big).
\end{align}
The claimed limiting distribution for $\hat{\theta}_n^{MPLE}$ follows form this, on setting $h=0$.

Proceeding to prove \eqref{eq:pl_gen}, using \eqref{PL_Optimizing}, we have
\begin{align*}
    \mathbf{X}^{T}Q_{n}\mathbf{X}=&\sum\limits_{i=1}^{n}t_i\tanh(\hat{\theta}_n^{MPLE}t_i)\\
    =&\sum_{i=1}^nt_i\tanh(\theta_{0} t_i)+\sum\limits_{i=1}^{n}t_i^2\text{sech}^{2}({\xi}_nt_i)(\hat{\theta}_{n}^{MPLE}-\theta_{0}),
\end{align*}
where ${\xi}_n$ lies between $\theta_{0}$ and $\hat{\theta}_{n}^{MPLE}$. The last display gives
\begin{equation}\label{eq:ratio_pl}
    \sqrt{n}(\hat{\theta}_{n}^{MPLE}-\theta_{0})=\frac{\frac{1}{\sqrt{n}}(\mathbf{X}^{T}Q_{n}\mathbf{X}-\sum\limits_{i=1}^{n}t_i\tanh(\theta_{0} t_i))}{\frac{1}{n}\sum\limits_{i=1}^{n}t_i^{2}\text{sech}^{2}({\xi}_nt_i)}.
\end{equation}

For analyzing the RHS of \eqref{eq:ratio_pl}, use part (a) to note that $\P_{\theta_0,Q_n}$ we have
\begin{align*}
  \log \frac{d\P_{\theta_n,Q_n}}{d\P_{\theta_0,Q_n}}({\mathbf X})\stackrel{d}{\to}N\Big(-\frac{h^2R(\theta_0)}{2}, h^2 R(\theta_0) \Big).
\end{align*}
Then, using Le Cam's third lemma, it follows that the measures $\P_{\theta_0,Q_n}$ and $\P_{\theta_n,Q_n}$ are mutually contiguous. Along with \cite[Lemma 2.1 part (a)]{deb2020fluctuations}, this gives that under $\P_{\theta_n,Q_n}$, we have
\begin{align*}
\Big(\sum_{i=1}^n\Big[t_i-m(\theta_0)\Big]^2|\bar{\mathbf X}>0\Big)=O_p(1).
\end{align*}

On the set $\bar{\mathbf X}>0$, under $\P_{\theta_n,Q_n}$ this gives
\begin{align*}
   & \sum\limits_{i=1}^{n}t_i\tanh(\theta_{0}t_i)\\
   =&\sum\limits_{i=1}^{n}\Big[m(\theta_0)\tanh\Big(\theta_{0} m(\theta_0)\Big)+m(\theta_0)\Big(1+\theta_{0} (1-m^{2}(\theta_0))\Big)\Big(t_i-m(\theta_0)\Big)\Big]+O_p\Big(\sum_{i=1}^n\Big(t_i-m(\theta_0)\Big)^2\Big)\\
   =&nm^2(\theta_0)+nm(\theta_0)\Big(1+\theta_0(1-m^2(\theta_0))\Big)\Big(\bar{\mathbf X}-m(\theta_0)\Big)+O_{p}(1).
\end{align*}
Since \eqref{eq:all_regimes} gives ${\mathbf X}^T Q_n{\mathbf X}=n\bar{\mathbf X}^2+O_p(1)$, the above display implies that under $\P_{\theta_n,Q_n}$, 
\begin{align*}
   \notag &\mathbf{X}^{T}Q_{n}\mathbf{X}-\sum\limits_{i=1}^{n}t_i\tanh(\theta_{0} t_i)\\
   =&n\Big(\bar{\mathbf X}^2-m^2(\theta_0)\Big)-nm(\theta_0)\Big(1+\theta_0(1-m^2(\theta_0))\Big)\Big(\bar{\mathbf X}-m(\theta_0)\Big)+O_p(1)\\
    =&n(\bar{\mathbf X}-m(\theta_0))\Big[\bar{\mathbf X}-\theta_0 m(\theta_0)(1-m^2(\theta_0))\Big]+O_p(1).
\end{align*}
Also, under $\P_{\theta_n,Q_n}$ we have
\begin{align}\label{eq:partcc}
    \sqrt{n}(\bar{\mathbf X}-m(\theta_0))=\sqrt{n}(\bar{\mathbf X}-m(\theta_n))+\sqrt{n}(m(\theta_n)-m(\theta_0))
    \stackrel{d}{\to}N(m(\theta_0)\sigma^2(\theta_0)h,\sigma^2),
\end{align}
where the last step uses part (a) of Lemma \ref{lem:mean} along with Proposition \ref{ppn:m}.
The last two displays give that under $\P_{\theta_n,Q_n}$,
\begin{align*}
    \frac{1}{\sqrt{n}}\Big({\mathbf X}^T Q_n{\mathbf X}-\sum_{i=1}^nt_i\tanh(\theta_0 t_i)\Big)\stackrel{d}{\to}\Big(1-\theta_0(1-m^2(\theta_0))\Big)N\Big(hm^2(\theta_0)\sigma^2(\theta_0), m^2(\theta_0)\sigma^2(\theta_0)\Big).
\end{align*}
Also, using contiguity of the two measures $\P_{\theta_0,Q_n}$ and $\P_{\theta_n,Q_n}$ along with consistency of $\hat{\theta}_n^{MPLE}$ gives $\xi_n\stackrel{p}{\to}\theta_0$ under both measures, and so
$$\frac{1}{n}\sum_{i=1}^nt_i^2\text{sech}^2(\xi_nt_i)\stackrel{p}{\to}m^2(\theta_0) \text{sech}^2\Big(\theta_0 m(\theta_0)\Big)=m^2(\theta_0)\Big(1-m^2(\theta_0)\Big).$$
Combining the last two displays along with \eqref{eq:ratio_pl} we get that under $\P_{\theta_n,Q_n}$,
\begin{align*}
\sqrt{n}(\hat{\theta}_n^{MPLE}-\theta_0)\stackrel{d}{\to}&\Big[\frac{1-\theta_0(1-m^2(\theta_0))}{m^2(\theta_0)(1-m^2(\theta_0))}\Big]N\Big(hm^2(\theta_0)\sigma^2(\theta_0),m^2(\theta_0)\sigma^2(\theta_0)\Big)\\
=&\frac{1}{m^2(\theta_0)\sigma^2(\theta_0)}N\Big(hm^2(\theta_0)\sigma^2(\theta_0),m^2(\theta_0)\sigma^2(\theta_0)\Big),
\end{align*}
which is equivalent to \eqref{eq:pl_gen}.
\end{itemize}

\textbf{\textit{Part (c):}}

\begin{itemize}
    \item {\bf MS Test}
    
   If ${\mathbf X}\sim \P_{\theta_0,Q_n}$, then \eqref{eq:partb} gives $$K_n(\alpha)=nm^2(\theta_0)+2z_\alpha \sqrt{nR(\theta_0)}+o(\sqrt{n}).$$
  Also, using \eqref{eq:partcc}, under $\P_{\theta_n,Q_n}$ we have
  $$\sqrt{n}\Big(\bar{\mathbf X}^2-m^2(\theta_0)\Big)\stackrel{d}{\to}2m(\theta_0) N\Big(m(\theta_0)\sigma^2(\theta_0) h,\sigma^2\Big)=N(2m^2(\theta_0)\sigma^2(\theta_0)h,4m^2(\theta_0)\sigma^2(\theta_0)\Big).$$
  Using the last two displays, we have
  \begin{align*}
      \P_{\theta_n,Q_n}(n\bar{\mathbf X}^2>K_n(\alpha))=&\P_{\theta_n,Q_n}\Big(\sqrt{n}(\bar{\mathbf X}^2-m^2(\theta_0))>2z_\alpha \sqrt{R(\theta_0)}\Big)+o(1)\\
      =&\P\Big(N(2R(\theta_0) h, 4R(\theta_0))>2z_\alpha \sqrt{R(\theta_0)}\Big)+o(1)\\
      =&\P\Big(N(h\sqrt{R(\theta_0)}, 1)>z_\alpha \Big)+o(1).
  \end{align*}
  Thus we have $\beta_{MS}=1-\Phi(z_\alpha-h\sqrt{R(\theta_0)})$, as claimed.
  \\
  
  \item{\bf NP Test}
  
  Note that $${\mathbf X}^T Q_n{\mathbf X}=n\bar{\mathbf X}^2+O_p(1),$$
  using Lemma \ref{lem:mean} part (a) under $\P_{\theta_0,Q_n}$ and  $\P_{\theta_n,Q_n}$. Thus ${\mathbf X}^T Q_n{\mathbf X}$ has the same asymptotic distribution as $n\bar{\mathbf X}^2$, under both null and the alternative. This gives $\beta_{NP}=\beta_{MS}$, as desired.
  \\
  
 \item
 {\bf PL Test}
 
 Using the limit distribution of $\hat{\theta}_n^{MPLE}$ in part (b) (invoke \eqref{eq:pl_gen} under $\P_{\theta_0,Q_n}$) gives
 $$K_n(\alpha)=\theta_0+\frac{z_\alpha}{\sqrt{nR(\theta_0)}}+o\Big(\frac{1}{\sqrt{n}}\Big).$$
 Thus we have
 \begin{align*}
     \P_{\theta_n,Q_n}(\hat{\theta}_n^{MPLE}>K_n(\alpha))=&\P_{\theta_n,Q_n}\Big(\sqrt{n}(\hat{\theta}_n^{MPLE}-\theta_0)>\frac{z_\alpha}{\sqrt{R(\theta_0)}}\Big)+o(1)\\
     =&\P\Big(N\Big(h,\frac{1}{R(\theta_0)}\Big)>\frac{z_\alpha}{\sqrt{R(\theta_0)}}\Big)+o(1),
 \end{align*}
 where the last step again uses \eqref{eq:pl_gen} under $\P_{\theta_n,Q_n}$. The last term in the display above converges to $$\P(N( h\sqrt{R(\theta_0)},1)>z_\alpha)=1-\Phi(z_\alpha-h\sqrt{R(\theta_0)}),$$
 thus giving the same formula for $\beta_{PL}$.

\end{itemize}


\subsection{Proof of Theorem \ref{test_critical}}

$ $\newline
\textbf{\textit{Part (a):}}

As in the proof of Theorem \ref{test_low}, let $I:=\{h_1,\ldots,h_k\}$ with $\{h_1<h_2<\ldots<h_k\}$ for some positive integer $k$, and let $\theta_{n,i}:=\theta_0+\frac{h_i}{\sqrt{n}}$ for $1\le i\le k$. It thus suffices to analyze the terms in the RHS of \eqref{eq:log1}. To this effect, use Lemma \ref{lem:normalizing_ising} part (b) to get
$$Z_n(\theta_{n,i},Q_n)-Z_n(\theta_0,Q_n)\to F(h)-F(0).$$
Proceeding to analyze the second term in the RHS of \eqref{eq:log1}, 
using \eqref{eq:all_regimes}, under $\P_{\theta_0,Q_n}$ we have
\begin{align}\label{eq:crit_m2}
\frac{{\mathbf X}^T Q_n{\mathbf X}}{\sqrt{n}}=\sqrt{n}\bar{\mathbf X}^2+O_p\Big(\frac{1}{\sqrt{n}}\Big)\stackrel{d}{\to}U_{0}^2, 
\end{align}
where the last step uses part (b) of Lemma \ref{lem:mean}. Combining the last two displays along with \eqref{eq:log1}, under $\P_{\theta_0,Q_n}$ we have
\begin{align*}
\Big[\log \frac{d\P_{\theta_{n,i},Q_n}}{d\P_{\theta_0,Q_n}}({\mathbf X})\Big]\stackrel{d}{\to}\Big[-F(h_i)+F(0)+\frac{h_i}{2}U_0^2\Big]_{1\le i\le k}.
\end{align*}
Also, we have
\begin{align*}
    \log \frac{d\H_{h_i}}{d\H_0}(u)=\frac{1}{2}h_iu-F(h_i)+F(0).
\end{align*}
Combining the above two displays, under $\P_{\theta_0,Q_n}$ we have
\begin{align*}
    \Big[\log \frac{d\P_{\theta_{n,i},Q_n}}{d\P_{\theta_0,Q_n}}({\mathbf X})\Big]\stackrel{d}{\to} \Big[\frac{d\H_{h_i}}{d\H_0}(U_0)\Big]_{1\le i\le k},
\end{align*}
where $U_0\sim \H_0$. This verifies the desired convergence of experiments.

\textbf{\textit{Part (b):}}

\begin{itemize}
    \item {\bf MLE}
    
    Existence of MLE follows from Proposition \ref{prop_exist} part (b).

Proceeding to find the limiting distribution, for any $h\in \R$ setting $\theta_n:=1+\frac{h}{\sqrt{n}}$, differentiating the second display in part (b) of Lemma \ref{lem:normalizing_ising} with respect to $h$ we get
\begin{align*}
\frac{2Z_n'(\theta_n,Q_n)}{\sqrt{n}}=\E U_h^2+o(1).
\end{align*}
Consequently, using calculations similar to the proof of Theorem \ref{test_low} part (b) we have
\begin{align*}
\P_{1,Q_{n}}\big(\sqrt{n}(\Tilde{\theta}_{n}^{MLE}-1)\leq h\big)
=&\P_{1,Q_{n}}\big(\mathbf{X}^{T}Q_{n}\mathbf{X}\leq 2Z_n'(\theta_n,Q_n)\big) \\
 =&\P_{1,Q_{n}}\big(\frac{1}{\sqrt{n}}\mathbf{X}^{T}Q_{n}\mathbf{X}\leq \E U_h^2+o(1)\big)\\
=& \P\big(U_0^{2}\leq\E U_h^{2} \big)+o(1),
\end{align*}
which derives the limiting distribution of the MLE.
\\

\item{\bf MPLE}

Existence of the MPLE follows from Proposition \ref{prop_exist} part (d).  Proceeding to show consistency, use \eqref{PL_Optimizing} to note that on the set $\hat{\theta}_n^{MPLE}>1+\delta$ under $\P_{1,Q_n}$ we have
\begin{align*}
    {\mathbf X}^T Q_n{\mathbf X}=&\sum_{i=1}^nt_i\tanh(\hat{\theta}_n^{MPLE} t_i)\\
    \ge &\sum_{i=1}^nt_i\tanh((1+\delta)t_i)\\
    =&n\bar{\mathbf X}\tanh((1+\delta)\bar{\mathbf X})+O_p\left(\sum_{i=1}^n(t_i-\bar{\mathbf X})^2\right).
    \end{align*}
    Using Lemma \ref{lem:mean} part (b), under $\P_{1,Q_n}$ we get
    $${\mathbf X}^TB_n{\mathbf X}=O_p(1), \quad \sum_{i=1}^n(t_i-{\mathbf X})^2={\mathbf X}^TB_n^2{\mathbf X}=O_p(1).$$
    Combining the last two displays, under $\P_{1,Q_n}$ we have
    $$\sqrt{n}\bar{\mathbf X}^2\ge \sqrt{n}\bar{\mathbf X}\tanh((1+\delta)\bar{\mathbf X})+O_p\Big(\frac{1}{\sqrt{n}}\Big).$$ But this is a contradiction, as the LHS of the above display converges in distribution $U_0^2$ under $\P_{1,Q_n}$ (by Lemma \ref{lem:mean} part (b)), and the RHS converges in distribution to $(1+\delta)U_0^2$ (which is larger).
Thus we have shown that for any $\delta>0$ we have 
$\P_{1,Q_n}(\hat{\theta}_n^{MPLE}\ge 1+\delta)\to 0.$
A similar proof gives $\P_{1,Q_n}(\hat{\theta}_n^{MPLE}\le 1-\delta)\to 0,$
and so $\hat{\theta}_n^{MPLE}\stackrel{p}{\to}1$.
\\

Proceeding to find the limiting distribution of $\hat{\theta}_n^{MPLE}$, setting $\theta_n:=1+\frac{h}{\sqrt{n}}$ for some $h\in \R$ we work under the measure $\P_{\theta_n,Q_n}$, and claim that
\begin{align}\label{eq:claim5}
\sqrt{n}(\hat{\theta}_n^{MPLE}-1)\stackrel{d}{\to}V_h.
\end{align}
As before, the desired limiting distribution for $\sqrt{n}(\hat{\theta}_n^{MPLE}-1)$ under $\P_{1,Q_n}$ then follows on taking $h=0$.
\\

For proving \eqref{eq:claim5}, use part (a) 
to conclude that the measures $\P_{1,Q_n}$ and $\P_{\theta_n,Q_n}$ are mutually contiguous  (as in the proof of Theorem \ref{test_low} part (b)), and so $\hat{\theta}_n^{MPLE}\stackrel{p}{\to}1$ under $\P_{\theta_n,Q_n}$ as well. Further,
using \eqref{PL_Optimizing}, under $\P_{\theta_n,Q_n}$ we have
\begin{align*}
&\sum_{i=1}^nt_i\tanh(\hat{\theta}_n^{MPLE} t_i)\\
    =&\sum_{i=1}^nt_i\Big[\tanh(t_i)+(\hat{\theta}_n^{MPLE}-1) t_i\text{sech}^2(t_i)+O_p\Big((\hat{\theta}_n^{MPLE}-1)^2|t_i|^3\Big)\Big].
\end{align*}
Under $\P_{\theta_n,Q_n}$, this gives
\begin{align}\label{eq:pl_crit}
    \hat{\theta}_n^{MPLE}-1=\frac{{\mathbf X}^T Q_n{\mathbf X}-\sum_{i=1}^nt_i\tanh(t_i)}{\sum_{i=1}^nt_i^2\text{sech}^2(t_i)+O_p((\hat{\theta}_n^{MPLE}-1)\sum_{i=1}^nt_i^4)}.
\end{align}
For analyzing the numerator in \eqref{eq:pl_crit}, under $\P_{\theta_n,Q_n}$ we have
\begin{align*}
   & {\mathbf X}^T Q_n{\mathbf X}-\sum_{i=1}^nt_i\tanh(t_i)\\
    =&n\bar{\mathbf X}^2+{\mathbf X}^TB_n{\mathbf X}-n\bar{\mathbf X}\tanh(\bar{\mathbf X})-\sum_{i=1}^n(t_i-\bar{\mathbf X})^2+O_p(\sum_{i=1}^n|t_i-\bar{\mathbf X}|^3)\\
    =&\frac{n\bar{\mathbf X}^4}{3}+{\mathbf X}^TB_n{\bf X}-{\mathbf X}^TB_n^2{\mathbf X}+O_p\left(n \sqrt{\frac{\log n}{n}}^3\right)+O_p\Big( n^{1-\frac{6}{4}}\Big)\\
    =&\frac{n\bar{\mathbf X}^4}{3}+{\mathbf X}^TB_n{\bf X}-{\mathbf X}^TB_n^2{\mathbf X}+o_p(1),
\end{align*}
where in the last but one step we use Lemma \ref{lem:mean} part (b) and \cite[Lemma 2.4]{deb2020fluctuations}, along with mutual contiguity shown above, to get that under $\P_{\theta_n,Q_n}$ we have $$|\bar{\mathbf X}|=O_p(n^{-1/4}),\quad \max_{i\in [n]}|t_i-\bar{\mathbf X}|=O_p\Big(\sqrt{\frac{\log n}{n}}\Big).$$ For the denominator in \eqref{eq:pl_crit}, under $\P_{\theta_n,Q_n}$ we have
\begin{align*}
   &\sum_{i=1}^nt_i^2\text{sech}^2(t_i)+O_p((\hat{\theta}_n^{MPLE}-1)\sum_{i=1}^nt_i^4)\\
   =&\sum_{i=1}^nt_i^2+O_p(\sum_{i=1}^nt_i^4)\\
   =&n\bar{\mathbf X}^2+\sum_{i=1}^n(t_i-\bar{\mathbf X})^2+O_p\Big( \sum_{i=1}^n(t_i-\bar{\mathbf X})^4+n\bar{\mathbf X}^4\Big)\\
   =&n\bar{\mathbf X}^2+O_p(1),
\end{align*}
where the last step uses Lemma \ref{lem:mean} part (a), and mutual contiguity. Combining the above two displays along with \eqref{eq:pl_crit}, under $\P_{\theta_n,Q_n}$ we get
\begin{align*}
\sqrt{n}(\hat{\theta}_n^{MPLE}-1)\stackrel{d}=&\frac{\frac{n\bar{\mathbf X}^4}{3}+{\mathbf X}^TB_n{\mathbf X}-\mathbf{X}^T B_n^2 {\mathbf X}+o_p(1)}{\sqrt{n}\bar{\mathbf X}^2+o_p(1)}\\
\stackrel{d}{\to}& \frac{\frac{U_h^4}{3}+S_1-T_1}{U_h^2},
\end{align*}
where we again use part (b) of Lemma \ref{lem:mean}. Recalling the formula of $V_h$ we have verified \eqref{eq:claim5}, and this completes part (b).

\end{itemize}

\textbf{\textit{Part (c):}}

\begin{itemize}
    \item{\bf MS Test}

    Using symmetry of the distribution of $\bar{\mathbf X}$ we have
    \begin{align*}
    \alpha=\P_{1,Q_n}(n\bar{\mathbf X}^2>K_n(\alpha))=2\P_{1,Q_n}(\bar{\mathbf X}>\sqrt{K_n(\alpha)}).
        \end{align*}
  Using this, along with the limit distribution  $n^{1/4}\bar{\mathbf X}\stackrel{d}{\to}U_0$ under $\P_{1,Q_n}$  (see part (b) of Lemma \ref{lem:mean}) gives
    $$\sqrt{K_n(\alpha)}=n^{1/4}\Psi_{\mathbb{H}_0}(1-\alpha/2)+o(n^{1/4}).$$
    Then, setting $\theta_n=1+\frac{h}{\sqrt{n}}$, the asymptotic power is given by
    \begin{align*}
        2\P_{\theta_n}(\sqrt{n}\bar{\mathbf X}>K_n(\alpha))=&2\P_{\theta_n,Q_n}\Big(n^{1/4}\bar{\mathbf X}>\Psi_{\mathbb{H}_0}(1-\alpha/2)+o(1)\Big)\\
        =&2\P\Big(U_h>\Psi_{\mathbb{H}_0}(1-\alpha/2)\Big)+o(1),
    \end{align*}
    where the last line again uses part (b) of Lemma \ref{lem:mean}.
    Thus we have $\beta_{MS}=2\P\Big(U_h>\Psi_{\mathbb{H}_0}(1-\alpha/2)\Big)$, as desired.

    \item{\bf NP Test}
    
    As in the proof of Theorem \ref{test_low}, we have ${\mathbf X}^T Q_n{\mathbf X}=n\bar{\mathbf X}^2+O_p(1)$ under both $\P_{1,Q_n}$ and $\P_{\theta_n,Q_n}$, using \eqref{eq:all_regimes} and 
    mutual contiguity. Thus we get $\beta_{NP}=\beta_{MS}$ as before.

\item{\bf PL Test}

Using \eqref{eq:claim5}, under $\P_{1,Q_n}$ we have 
$$\sqrt{n}(\hat{\theta}_n^{MPLE}-1)\stackrel{d}{\to}V_{0}, \text{ which gives }K_n(\alpha)=1+\frac{\Psi_{V_{0}}(\alpha)}{\sqrt{n}}+o\Big(\frac{1}{\sqrt{n}}\Big).$$
Then the asymptotic power is given by
\begin{align*}
    \P_{\theta_n,Q_n}(\hat{\theta}_n>K_n(\alpha))=\P(V_h>\Psi_{V_{0}}(\alpha))+o(1),
\end{align*}
where the last step again uses \eqref{eq:claim5}. This shows that $\beta_{PL}=\P(V_h>\Psi_{V_{0}}(\alpha))$, as desired.

\end{itemize}

   \subsection{Proof of Proposition \ref{prop_exist}}
 
%
%
%
%
%
%
%
%
%
%
%

\begin{enumerate}
   \item[(a)]
By definition of $a_n, b_n$ we have
$$\lim_{\theta\to -\infty}\E_{\theta}{\mathbf X}'Q_n{\mathbf X}=a_n,\quad \lim_{\theta\to \infty}\E_{\theta}{\mathbf X}^T Q_n{\mathbf X}=b_n.$$
Since the function $$\theta\mapsto \frac{\theta}{2}{\mathbf X}^T Q_n{\mathbf X}-Z_n(\theta,Q_n)$$
is strictly concave, it follows that there exists a unique MLE in $\R$ iff the equation
$${\mathbf X}'Q_n{\mathbf X}=\E_{\theta}{\mathbf X}'Q_n{\mathbf X}$$
has a real solution, which holds iff
$ a_n<{\mathbf X}^T Q_n{\mathbf X}<b_n$, as desired.

\item[(b)]
This is immediate from part (a), and on noting that ${\mathbf X}^T Q_n{\mathbf X}$ has a continuous limiting distribution in all regimes, as shown in the proofs above (in particular, see \eqref{eq:claim2} and \eqref{eq:crit_m2} for domains $\Theta_1$ and $\Theta_2$ respectively).

 \item[(c)]
 Using \eqref{PL_Optimizing}, the existence of MPLE is equivalent to the existence of a real valued root of the equation
\begin{align}\label{eq:pl_again} {\mathbf X}^T Q_n{\mathbf X}=\sum_{i=1}^nt_i\tanh(\theta t_i).
\end{align}
 Taking limits as $\theta\to\pm\infty$, we get
$$\lim_{\theta\to-\infty}\sum_{i=1}^nt_i\tanh(\theta t_i)=-\sum_{i=1}^n|t_i|,\quad \lim_{\theta\to\infty}\sum_{i=1}^nt_i\tanh(\theta t_i)=\sum_{i=1}^n|t_i|.$$
Thus the existence of MPLE holds iff
$$-\sum_{i=1}^n|t_i|<{\mathbf X}^T Q_n{\mathbf X}<\sum_{i=1}^n|t_i|.$$
Suppose ${\mathbf X}^T Q_n{\mathbf X}=\sum_{i=1}^nt_i$. This happens iff $X_i=1$ for all $i\in S({\mathbf X})$. Similarly we have
$${\mathbf X}^T Q_n{\mathbf X}=-\sum_{i=1}^nt_i\Leftrightarrow X_i=-1\text{ for all }i\in S({\mathbf X}).$$
The conclusion of part (c) follows from this.

 \item[(d)]
 By symmetry, we only show that
 $$\P_{\theta_0,Q_n}(X_i=1, i\in S)\to 0.$$

 To this end, note that by mutual contiguity it suffices to show the result under the Curie-Weiss model.  To this effect, we first claim that for any positive sequence $\{\varepsilon_n\}_{n\ge 1}$ converging to $0$ and constant $C>0$ free of $n$, we have
 \begin{align}\label{eq:claim}
 \lim_{n\to\infty}\sup_{{\bf a}\in \R^n/\{0\}: \|{\bf a}\|_\infty\le \varepsilon_n  \|{\bf a}\|_2} \P(\sum_{i=1}^n a_i \xi_i=0)= 0,
 \end{align}
 where $\{\xi_i\}_{1\le i\le n}$ are iid random variables such that $\Var( \xi_1)\ne 0, \E |\xi_1|^3\le C$.
 
 Given this claim, we first complete the proof of part (d). Let $\phi_n$ be the auxiliary variable introduced in Proposition . Then we have
 \begin{align*}
     \P_{\theta_0,{\rm CW}}(t_i=0\text{ for some i}, 1\le i\le n)=\E \P_{\theta_0,{\rm CW}}(t_i=0\text{ for some i}, 1\le i\le n|\phi_n).
 \end{align*}
Given $\phi_n$, $t_i$ is a weighted sum of iid random random variables, with
 $\E(|X_1|^3|\phi_n)\le 1.$
 Also, we have
 $$|Q_n(i,j)|\le \frac{C_w}{n}, \quad \sqrt{\sum_{j=1}^nQ_n(i,j)^2}=\frac{1}{\sqrt{d_n}},$$
 and so we can take $C=1, \varepsilon_n=\frac{C'}{\sqrt{n}}$ for some suitable constant $C'$ free of $n$. Thus we have $$ \P_{\theta_0,{\rm CW}}(t_i=0\text{ for some i}, 1\le i\le n|\phi_n)\stackrel{p}{\to}0\Rightarrow \lim_{n\to\infty}\P_{\theta_0,{\rm CW}}(t_i=0\text{ for some }i)=0.$$  To complete the proof, it suffices to show that $\P_{\theta_0,{\rm CW}}(\bar{\mathbf X}={\bf 1})\to 0$. But this is immediate from the weak law of $\bar{\mathbf X}$ derived in Lemma \ref{lem:mean} (and mutual contiguity of $\P_{\theta_0,Q_n}$ and $\P_{\theta_0,{\rm CW}}$).
\\

It thus remains to verify the claim. But this follows on setting $\mu_n:=\E\xi_1, \tau_n^2=Var(\xi_1)$ and noting that
$$\P(\sum_{i=1}^na_i\xi_i=0)=\P\Big(\frac{\sum_{i=1}^na_i(\xi_i-\mu_n)}{\tau_n \|{\bf a\|_2^2}}=-\frac{\mu_n\sum_{i=1}^na_i}{\tau_n\|{\bf a}\|_2}\Big),$$
and the fact that 
$$\frac{\sum_{i=1}^na_i(\xi_i-\mu_n)}{\tau_n \|{\bf a}\|_2^2}\stackrel{d}{\to}N(0,1)$$
by the Lyapunov CLT.
    
\end{enumerate}

\begin{acks}[Acknowledgments]
We thank Nabarun Deb and Rajarshi Mukherjee for helpful comments throughout this project. We also thank Richard Nickl for suggesting this problem. The presentation of the paper greatly benefitted from the suggestions of an anonymous referee, and the Associate Editor.
\end{acks}

\begin{funding}
The second author was supported in part by NSF (DMS-2113414).
\end{funding}

\bibliographystyle{imsart-nameyear.bst}
\bibliography{main_sumit.bib}

\section{Appendix}\label{sec:appen}

The appendix is organized as follows:  In section \ref{sec:A} we prove some general results. In section \ref{sec:B} we use the results from \ref{sec:A} to prove two auxiliary lemmas. Finally in section \ref{sec:C} we use the two auxiliary results from section \ref{sec:B} to verify Lemma \ref{lem:mean} and Lemma \ref{lem:normalizing_ising}.

\subsection{Proofs of Independent Results}\label{sec:A}

\subsubsection{Limit distribution for IID quadratic forms}

We first state a proposition connecting eigenvalues of the matrix $Q_n$ and eigenvalues of the limiting graphon $f$.

\begin{ppn}\label{ppn:graphon}
Let $\{Q_{n}\}_{n=1}^{\infty}$ be a sequence of matrices satisfying \eqref{RIC}, \eqref{CWC}, \eqref{eq:frob} and \eqref{eq:cut} for some $C_W, \kappa\in (0,\infty)$ and $f\in \mathcal{W}$. Let $\{\lambda_{j,n}\}_{j=1}^{n}$ denote the eigenvalues of $Q_n$ arranged in decreasing order of absolute value, and let  
$\{\lambda_{j}\}_{j\geq1}$ be the eigenvalues of the operator $T_f$  as defined in section \ref{sec:graphon}. Then the following conclusions hold:
\begin{enumerate}
    \item[(a)]
    \begin{equation*}
        \sum\limits_{j=1}^{\infty}\lambda_{j}^{2}=\int_{[0,1]^{2}}f(x,y)^{2}dxdy=||f||_{2}^{2}<\infty;
    \end{equation*}
    \item[(b)] For any $j\in\N$, 
    \begin{equation*}
        \lim\limits_{n\rightarrow\infty}\lambda_{j,n}=\lambda_{j}.
    \end{equation*}

    \item[(c)] For any $i\geq3$ we have
   $$\lim\limits_{n\rightarrow\infty}\sum\limits_{j=2}^{n}\lambda_{j,n}^i = \sum\limits_{j=2}^{\infty}\lambda_{j}^{i}.$$
\end{enumerate}
\end{ppn}
\begin{proof}
The proof of Proposition \ref{ppn:graphon} part (a) follows \cite[Chapter 7.5]{lovasz2012large}, whereas part (b) and (c) follow from \cite[Theorem 11.54]{lovasz2012large}.
\end{proof}

{Utilizing Proposition \ref{ppn:graphon}, we now characterize the limit distribution of quadratic forms of IID random variables, which is a crucial result in our analysis, and is of possible independent interest. For interested readers, we refer to \cite[Theorem 1.4]{bhattacharya2017universal} and \cite[Theorem 1.4]{bhattacharya2022asymptotic} for results with a similar flavor.}
 
\begin{lem}\label{lem:comparison}

Suppose ${\mathbf Z}:=(Z_i)_{1\le i\le n}$ are IID random variables with mean $0$ and variance $\tau_n^2$ which converges to $\tau^2\in (0,\infty)$. Assume that the matrix $Q_n$ satisfies \eqref{RIC}, \eqref{CWC}, \eqref{eq:frob} and \eqref{eq:cut} for some $C_W, \kappa\in (0,\infty)$ and $f\in \mathcal{W}$. 
Then with $B_n=Q_n-\frac{1}{n}{\bf 1}{\bf 1}^T$ we have
$$[\sqrt{n}\bar{\mathbf Z}, {\mathbf Z}^TB_n{\mathbf Z}, \mathbf {Z}^TB_n^2{\mathbf Z}]\stackrel{d}{\to}\Big[\tau W_0, \tau^2\Big(\sum_{j=2}^\infty \lambda_j(Y_j-1)-1+W^*\Big),\tau^2 \Big(\sum_{j=2}^\infty \lambda_j^2 Y_j+\kappa\Big)\Big],$$
where
$$ W_0\sim N(0,1),\quad  W^*\sim N(0,2\kappa), \quad \{Y_j\}_{j\ge 2}\stackrel{iid}{\sim}\chi_1^2$$
are mutually independent. Here the infinite sums in the limiting distributions converge in $L_2$.

\end{lem}

\begin{proof}
The proof of Lemma \ref{lem:comparison} will be completed, once we can show the following two steps:

\begin{enumerate}
    \item[(a)] Suppose $\{R_i\}_{1\le i\le n}$ are IID $N(0,1)$. Then the desired conclusion holds.
    
    \item[(b)]
    For any positive integers $a,b,c$ we have
    $$\E \Big(\sqrt{n}\bar{\mathbf Z}\Big)^a \Big({\mathbf Z}^TB_n{\mathbf Z}\Big)^b\Big({\mathbf Z}^TB_n^2{\mathbf Z}\Big)^c-\E \Big(\sqrt{n}\bar{\mathbf R}\Big)^a \Big({\mathbf R}^TB_n{\mathbf R}\Big)^b\Big({\mathbf R}^TB_n^2{\mathbf R}\Big)^c\to 0.$$
\end{enumerate}

\begin{proof}[Proof of (a)]
Let $$Q_n=P^T\Lambda P=\sum_{i=1}^n\lambda_{i,n}{\bf p}_i{\bf p}_i^T$$ be the spectral decomposition of $Q_n$, where the eigenvalues $\{\lambda_{i,n}\}_{1\le i\le n}$ are arranged in decreasing order of absolute values. Thus we have $\lambda_{1,n}=1$, and ${\bf p}_1=\frac{1}{\sqrt{n}}{\bf 1}.$ Then setting $\widetilde{\mathbf R}:=P{\mathbf R}$ we have
$$\left[\sqrt{n}\bar{\mathbf R}, {\mathbf R}^TB_n{\mathbf R},{\mathbf R}^TB^2_n{\mathbf R}\right]=\left[\widetilde{R}_1,\sum_{i=2}^n\lambda_{i,n}\widetilde{R}_i^2, \sum_{i=2}^n\lambda^2_{i,n}\widetilde{R}_i^2 \right].$$
Since $\widetilde{\mathbf{R}}\stackrel{d}{=}\mathbf{R}$, it suffices to find the limiting distribution of
$$\left[{R}_1,\sum_{i=2}^n\lambda_{i,n}{R}_i^2, \sum_{i=2}^n\lambda^2_{i,n}{R}_i^2 \right]. $$
Clearly, $R_1$ is independent of the other two random variables, and has a $N(0,1)$ distribution. It thus suffices to focus on the joint distribution of the other two random variables. To this effect, for any $s,t$ with $\max(|s|,|t|)\le \frac{1}{8}$ we have
\begin{align}\label{eq:trivial0}
   \notag &\log \E \exp\Big\{s\sum_{j=2}^n\lambda_{j,n}{R}_j^2+t \sum_{j=2}^n\lambda^2_{j,n}{R}_i^2  \Big\}\\
   \notag=&-\frac{1}{2}\sum_{j=2}^n\log \Big[1-2(\lambda_{j,n}s+\lambda_{j,n}^2t)\Big]\\
   \notag =&\frac{1}{2}\sum_{j=2}^n\sum_{i=1}^\infty \frac{2^i(\lambda_{j,n}s+\lambda_{j,n}^2t)^i}{i}\\
   \notag =&\Big[-s+t \sum_{j=2}^n\lambda_{j,n}^2 \Big]+\frac{1}{2}\sum_{j=2}^n\sum_{i=2}^\infty \frac{2^i(\lambda_{j,n}s+\lambda_{j,n}^2t)^i}{i}\\
    =&\Big[-s+t \sum_{j=2}^n\lambda_{j,n}^2 \Big]+\frac{1}{2}\sum_{i=2}^\infty\sum_{j=2}^n \frac{2^i(\lambda_{j,n}s+\lambda_{j,n}^2t)^i}{i},
    \end{align}
    where the last line uses Fubini's theorem, along with the trivial bound
\begin{align}\label{eq:trivial}2^i(\lambda_{j,n}s+\lambda_{j,n}^2t)^i\le 4^i |\lambda_{j,n}|^i 8^{-i}\le 2^{-i} \lambda_{j,n}^2.  
\end{align}
For every fixed $i\ge 2$,
uses Proposition \ref{ppn:graphon} part (b) we have
$$\sum_{j=2}^n(\lambda_{j,n}s+\lambda_{j,n}^2t)^i\to\sum_{j=2}^\infty(\lambda_{j}s+\lambda_{j}^2t)^i.$$
Also, using \eqref{eq:trivial} we have 
$$\sum_{j=2}^n \frac{2^i(\lambda_{j,n}s+\lambda_{j,n}^2t)^i}{i}\le 2^{-i} \sum_{j=2}^n\lambda_{j,n}^2,$$
where 
$$\sum_{j=2}^n\lambda_{j,n}^2\to \sum_{j=2}^\infty \lambda_j^2<\infty$$
by Proposition \ref{ppn:graphon} part (c). Combining the last three displays along with dominated convergence theorem, the RHS of \eqref{eq:trivial0} converges to
    \begin{align*}
    \Big[-s+t \sum_{j=2}^\infty\lambda_{j}^2 \Big]+\frac{1}{2}\sum_{j=2}^\infty\sum_{i=2}^\infty \frac{2^i(\lambda_{j}s+\lambda_{j}^2t)^i}{i}.
\end{align*}
This is the log moment generating function of $$(\sum\limits_{j=2}^{\infty}\lambda_{j}(Y_{j}-1)-1+W^*,\sum\limits_{j=2}^{\infty}\lambda_{j}^{2}Y_{j}+\kappa).$$
The random variables in the RHS above converge in $L_2$, as 
\begin{align*}
\E \Big[\sum_{j=k+1}^\infty \lambda_j(Y_j-1)\Big]^2=&2\sum_{j=k+1}^\infty \lambda_j^2\stackrel{k\to\infty}{\to}0, \\
\E \Big[\sum_{j=k+1}^\infty \lambda_j^2 Y_j\Big]^2\le &3\sum_{j=k+1}^\infty \lambda_j^4+\Big(\sum_{j=k+1}^\infty \lambda_{j}^2\Big)^2\stackrel{k\to\infty}{\to}0.
\end{align*}
The convergence in the above display uses Proposition \ref{ppn:graphon} part (a). The proof of part (a) is complete. 
\end{proof}

\begin{proof}[Proof of (b)]
To begin, use \eqref{CWC} to note that
$|B_n(i,j)|\le \frac{C_W}{n}$, and
$$|B_n^2(i,j)|\le \sum_{k=1}^n|B_n(i,k)B_n(k,j)|\le \frac{C_W^2}{n}.$$
Set $r:=\frac{a}{2}+b+c$, and let $\mathcal{S}(\ell,2r)$ denote the set of all positive integer solutions to the equation $\sum_{i=1}^\ell\alpha_i=2r$. Then  we have
\begin{align}\label{eq:graph_bound}
 \notag  &\Big| \E(\sqrt{n}{\mathbf Z})^a({\mathbf Z}^TB_n{\mathbf Z})^b ({\mathbf Z}^TB_n^2{\mathbf Z})^c-\E(\sqrt{n}{\mathbf R})^a({\mathbf R}^TB_n{\mathbf R})^b ({\mathbf R}^TB_n^2{\mathbf R})^c\Big|\\
   \le& n^{-r} C_W^{b+2c}\sum_{\ell=1}^{2r} n^{\ell} \sum_{{\bm \alpha}\in \mathcal{S}(\ell,2r)} \Big|\E \prod_{i=1}^\ell Z_i^{\alpha_i}-\E \prod_{i=1}^\ell R_i^{\alpha_i}\Big|.
\end{align}
To bound the RHS of \eqref{eq:graph_bound}, we consider the following cases based on ${\bm \alpha}$

\begin{itemize}
    \item {There exists $i\in [\ell]$ such that $\alpha_i=1$}
    
    In this case we have
    $$\E \prod_{i=1}^\ell Z_i^{\alpha_i}=\E \prod_{i=1}^\ell R_i^{\alpha_i}=0.$$

    \item {$\ell>r$}
    
    In this case we claim that there exists $i\in [\ell ]$ such that $\alpha_i=1$. Thus this is a sub case of the above case.
    
    Suppose not.  Then we have
    $$2r=\sum_{i=1}^\ell \alpha_i\ge 2\ell,$$
    which is a contradiction.
    
    \item{$\ell=r, \alpha_i\ge 2\text{ for all }i\in [\ell]$}
    
    In this case we must have $\alpha_i=2$ for all $i$. If not, then we must have
    $$2r=\sum_{i=1}^\ell \alpha_i> 2\ell,$$
    a contradiction. Since $\E Z_i^2=\E R_i^2=1$, we have  $$\E \prod_{i=1}^\ell Z_i^{\alpha_i}=\E \prod_{i=1}^\ell R_i^{\alpha_i}=1.$$
    
\end{itemize}
Combining the cases above, the RHS of \eqref{eq:graph_bound} is bounded by
\begin{align*}  n^{-r} C_W^{b+2c}\sum_{\ell=1}^{r-1} n^{\ell} \sum_{{\bm \alpha}\in \mathcal{S}(\ell,2r)} \Big|\E \prod_{i=1}^\ell Z_i^{\alpha_i}-\E \prod_{i=1}^\ell R_i^{\alpha_i}\Big|=O\Big(\frac{1}{n}\Big), 
\end{align*}
and so the proof of part (b) is complete.
\end{proof}

\end{proof}

\subsubsection{Curie-Weiss model}\label{sec:ising}
As it turns out, our proof technique relies on a very precise understanding of what happens under the Curie-Weiss model $\P_{\theta, {\rm CW}}$, defined in \eqref{eq:cw}, and recalled here below: 
\begin{align*}
\P_{\theta, {\rm CW}}({\mathbf X}={\mathbf x})=\exp\Big( \frac{n\theta\bar{x}^2}{2}-Z_n(\theta,{\rm CW})\Big).
\end{align*}

The following proposition expresses the Curie-Weiss model as a mixture of IID laws. The same decomposition was also utilized previously in the literature (see \cite{deb2020fluctuations,mukherjee2018global}. We omit the proof.

\begin{ppn}\label{ppn_aux}\cite[Lemma 3]{mukherjee2018global}, \cite[Proposition 4.1]{deb2020fluctuations}
Given ${\mathbf X}\sim \P_{\theta, {\rm CW}}$, let $\phi_n$ be a real valued random variable defined by 
\begin{equation}\label{phi}
\phi_{n}\sim N\Big(\Bar{\mathbf{X}},\frac{1}{n\theta}\Big).
\end{equation}
Then the following conclusions hold:
\begin{enumerate}
    \item[(a)] Given $\phi_{n}$, the random variables $(X_1,\ldots,X_n)$ are IID, with 
    \begin{equation}
        \P_{\theta, {\rm CW}}(X_{j}=1|\phi_{n}) = \frac{\exp{(\theta \phi_{n})}}{\exp{(\theta \phi_{n})}+\exp{(-\theta \phi_{n})}}
    \end{equation}
    \item[(b)] The marginal density of $\phi_n$ has a density with respect to Lebesgue measure, which is proportional to 
    \begin{equation}
 f_{\theta,n}(\phi_{n})=\exp{\Big(-\frac{1}{2}n\theta \phi_{n}^{2}+n\log \cosh(\theta \phi_{n})\Big)}.
\end{equation}
\end{enumerate}
\end{ppn}

\begin{defn}\label{def:phi}
Let $F_{n,\theta}$ denote the distribution of $\phi_n$, as defined in Proposition \ref{ppn_aux}.
\end{defn}
We begin by proving a lemma about the distribution $F_{n,\theta}$. 

\begin{lem}\label{lem:phi1}
Fix $\theta_0>0, h\in \R$, and let $\theta_n=\theta_0+\frac{h}{\sqrt{n}}$ be as in definition \ref{def:thetan}. Let $\phi_n\sim F_{n,\theta_n}$, where $F_{n,\theta}$ is as in definition \ref{def:phi}. 
\begin{enumerate}
    \item[(a)]
    If $\theta_0\in \Theta_3$, then 
    $$\sqrt{n}\phi_n\to N\Big(0, \frac{1}{\theta_0-\theta_0^2}\Big)$$
    in distribution, and in moments.

     \item[(b)]
    If $\theta_0\in \Theta_2$, then 
    $$n^{1/4}\phi_n\to \H_h$$
    in distribution, and in moments, where $H_h$ is as defined in \eqref{H_h}.

  \item[(c)]
  If $\theta_0\in \Theta_1$, then conditional on $\phi_n>0$ we have
    $$\sqrt{n}(\phi_n-m(\theta_n))\to N\Big(0, \frac{1}{\theta_0-(1-m^2(\theta_0))\theta_0^2}\Big)$$
    in distribution, and in moments, where $m(.)$ is as in definition \ref{Three_domains}.
 
\end{enumerate}

\end{lem}

\begin{proof}

\begin{enumerate}
    \item[(a)] \textbf{High Temperature Regime} $\Theta_{3}$. 
    \\
    
    Use Proposition \ref{ppn_aux} to note that $\phi_{n}$ has a density proportional to 
    \begin{equation}
    f_{\theta_{n},n}(\phi)=\exp\{-nq_{\theta_n}(\phi)\},\quad  q_{\theta}(\phi)=\frac{1}{2}\theta\phi^{2}-\log cosh(\theta \phi).
    \end{equation}
    Differentiating twice we get
    $$\theta_n\ge q_{\theta_n}''(\phi)\ge \theta_n-\theta_n^2,$$
    and so if $\theta_0\in \Theta_3$, there exists finite positive constants $c_1,c_2$ depending only on $\theta_0,h$ (and free of $n$), such that for all $n$ large enough we have
    $$c_1\le q_{\theta_n}''(\phi)\le c_2.$$
    Consequently, for any $\phi$ we have
    $$\frac{c_1}{2}\phi^2\le q_{\theta_n}(\phi)\le \frac{c_2}{2}\phi^2,$$
   and so for any $K>0$ we have
    \begin{align*}
        \P(\sqrt{n}|\phi_n|>K)\le \frac{2\int_{K}^\infty e^{-n q_{\theta_n}(\phi)}d\phi}{\int_{-\infty}^\infty  e^{-nq_{\theta_n}(\phi)}d\phi}
        \le \frac{2\int_K^\infty e^{-c_1\phi^2/2}d\phi}{\int_{-\infty}^\infty e^{-c_2 \phi^2/2}d\phi}
        \le 2\sqrt{\frac{c_2}{c_1}} \P\Big(N(0, \frac{1}{c_1})>K\Big).
    \end{align*}
    Thus we have $\sqrt{n}\phi_n=O_p(1)$, and further all moments of $\sqrt{n}\phi_n$ are bounded. Finally, since $q_{\theta_n}''(\phi)\to\theta_0-\theta_0^2,$ it follows from standard calculus that for any $a,b$ fixed, standard calculus gives
    \begin{align*}
       \sqrt{n} \int_{a/\sqrt{n}}^{b/\sqrt{n}}e^{-nq_{\theta_n(\phi)}}d\phi\to \int_a^b e^{-\frac{\theta_0-\theta_0^2}{2}t^2}dt.    \end{align*}
    Combining the above calculations, it follows that $$\sqrt{n}\phi_n\to N\Big(0,\frac{1}{\theta_0-\theta_0^2}\Big),$$
    in distribution and in moments.
    \\
    
    \item[(b)] \textbf{Critical Regime} $\Theta_{2}$. 
    \\
    
   In this case we have
   \begin{align*}
   q_{\theta_n}'(0)=q_{\theta_n}'''(0)=0,\quad
       \sqrt{n}q_{\theta_n}''(0)=\sqrt{n}(\theta_n-\theta_n^2)\to-h,\quad
      c_1'\le  \inf_{|\phi|\le 2}q_{\theta_n}''''(\phi)\le \sup_{|\phi|\le 2}q_{\theta_n}''''(\phi)\le c_2',
   \end{align*}
   for some constants $c_1', c_2'$ depending only on $h$. Also, using Proposition \ref{ppn_aux} we have 
   $$\P(|\phi_n|>2)\le \P(|N(0,1)|>\sqrt{n\theta_n}),$$
   which is exponentially small in $n$. The above two displays together give $n^{1/4}\phi_n=O_p(1)$. Finally, fixing $a,b$,  straight-forward calculus gives
   \begin{align*}
       \int_{a/n^{1/4}}^{b/n^{1/4}} n^{1/4} e^{-nq_{\theta_n}(\phi)}d\phi\to \int_a^b e^{\frac{h}{2}\phi^2-\phi^4/12 }d\phi,
   \end{align*}
   where we use the fact that
   $q_{\theta_n}''''(0)\to 2.$
   Combining, the desired limiting distribution follows. Uniform integrability also follows from the estimates on $q_{\theta_n}(.)$.
   \\
   
   \item[(c)] \textbf{Low Temperature Regime} $\Theta_{1}$.
   \\
   
   In this case we have $\P(\phi_n>0)=\P(\phi_n<0)=\frac{1}{2}$ by symmetry. Restricting on the positive half without loss of generality, note that the function $q_\theta(\phi)$ has a unique minimizer in $\phi$ on $(0,\infty)$, at the point $m(\theta)$. From Proposition \ref{ppn_aux} we have $q''(m(\theta))>0$, and so the function $\Psi:[0,2]\times [0,1]$ defined by
   \begin{align*}
  \Psi(x,\theta):=& \frac{q_\theta(x)-q_\theta(m(\theta))}{(x-m(\theta))^2}\text{ if }x\ne m(\theta),\\
  =&\frac{q''_\theta(m(\theta))}{2}\text{ if }x=m(\theta),
  \end{align*}
   is strictly positive and continuous, and so there exists finite positive constants $c_1,c_2$ depending on $\theta_0,h$, such that for all $\phi>0$ we have
   $$\frac{c_1}{2}\phi^2\le q_{\theta_n}(\phi)-q_{\theta_n}(m({\theta_n}))\le \frac{c_2}{2}\phi^2.$$
   From this, a similar calculation as in part (a) of this lemma applies, on noting that
   $$q''_{\theta_n}(m({\theta_n}))\stackrel{p}{\to} \theta_0-\theta_0^2\Big(1-m^2({\theta_0})\Big). $$
    
\end{enumerate}

\end{proof}

\subsubsection{Two analysis results}

In this section, we first verify Proposition \ref{ppn:m}, which was used to prove the main results, and will be used in the sequel as well.

\begin{proof}[Proof of \ref{ppn:m}]

We prove the more general result
$$\lim_{h\to 0}\frac{m(\theta_0+h)-m(\theta_0)}{h}=m(\theta_0)\sigma^2(\theta_0).$$
The desired conclusion then follows on replacing $h$ by $\frac{h}{\sqrt{n}}$, and letting $n\to\infty$. Recall that $m(\theta)$ satisfies the equation $w(\theta,m)=0$ in $m$, where
$$w(\theta,m):=m-\tanh(\theta m).$$
Differentiating with respect to $\theta$ we get
$$\frac{\partial w(\theta,m)}{\partial \theta}=1-\theta \text{sech}^2(\theta m).$$
By Proposition \ref{prop_exist}, we have $\theta(1-m^2(\theta))<1$, and so the above derivative is always positive. By Implicit Function Theorem, it follows that the function $\theta\mapsto m(\theta)$ is differentiable. On differentiating the equation 
$$m(\theta)=\tanh(\theta m(\theta))$$
with respect to $\theta$, we get
$$m'(\theta)=\text{sech}^2(\theta m(\theta))[m(\theta)+\theta m'(\theta)]=(1-m^2(\theta))[m(\theta)+\theta m'(\theta)].$$
Solving for $m'(\theta)$ gives
$$m'(\theta)=\frac{m(\theta)(1-m^2(\theta)}{1-\theta(1-m^2(\theta))}=m(\theta) \sigma^2(\theta),$$
as desired.

\end{proof}

Finally, we prove a uniform convergence result for monotone functions. 
\begin{ppn}\label{ppn:monotone}
Suppose $g_n(.)$ is a sequence of functions defined on a compact interval $[a,b]$. Assume the following:
\begin{itemize}
    \item There exists a function $g_\infty(.)$ on $[a,b]$, such that $g_n(.)$ converges to $g_\infty(.)$ pointwise.
    
   \item The function $g_n(.)$ is non-decreasing, for every $n\ge 1$.
   
   \item
   The function $g_\infty(.)$ is continuous.
    
\end{itemize} 

Then $g_n$ converges to $g_\infty$ uniformly on $[a,b]$.
\end{ppn}

\begin{proof}

Let $\{t_n\}_{n\ge 1}$ be a real sequence in $[a,b]$ converging to $t_\infty$. We need to show that $g_n(t_n)$ converges to $g_\infty(t_\infty)$. To this effect, fixing $\delta>0$ arbitrary, for all $n$ large we have
$|t_n-t_\infty|<\delta$. Using the monotonicity of $\{g_n\}_{1\le n\le \infty}$ gives $$g_n(t_n)-g_\infty(t_\infty)\le g_n(t-\delta)-g_\infty(t_\delta).$$
Taking limits as $n\to\infty$ gives
$$\limsup_{n\to\infty}\{g_n(t_n)-g_\infty(t_\infty)\}\le g_\infty(t_\infty+\delta)=g_\infty(\delta).$$
Since $\delta$ is arbitrary, letting $\delta\downarrow 0$ and using the continuity of $g_\infty(.)$ gives
$$\limsup_{n\to\infty}\{g_n(t_n)-g_\infty(t_\infty)\}\le 0.$$
A similar proof gives
$$\liminf_{n\to\infty}\{g_n(t_n)-g_\infty(t_\infty)\}\ge 0.$$
The proof is complete by combining the last two displays.

\end{proof}


\subsection{Two auxiliary lemmas}\label{sec:B}

We now utilize the results of section \ref{sec:A} to state and prove two auxiliary lemmas under the Curie-Weiss model, which we will use to verify Lemma \ref{lem:mean} and Lemma \ref{lem:normalizing_ising}.

\begin{lem}\label{XBX_CW}
For any $h\in\R$ and $\theta_0>0$, let $\theta_n=\theta_0+\frac{h}{\sqrt{n}}$ be as in definition \ref{def:thetan}.
Let $\mathbf{X}\sim\P_{\theta_n, {\rm CW}}$, where $Q_{n}$ is a sequence of matrices which satisfy \eqref{RIC}, \eqref{CWC}, \eqref{eq:frob} and \eqref{eq:cut}. Set $B_{n} = Q_{n}-\frac{1}{n}\bf{1}\bf{1}^{T}$ as before. Also, let $W_{\theta_0}\sim N(0,\sigma^2(\theta_0))$, $U_h\sim \mathbb{H}_h$ (see \eqref{H_h}) be independent of $(S_0,T_0)$ (see \eqref{eq:ss} and \eqref{eq:tt} respectively). Then the following conclusions hold under $\P_{\theta_n,{\rm CW}}$:
\begin{enumerate}
    \item[(a)]\textbf{Low Temperature Regime}:\label{XBX_a1}: $\theta_{0}\in \Theta_1$. 
   \\
   
    \begin{enumerate}
        \item[(i)]
  With $m(.)$ as in definition \ref{Three_domains}, we have 
    \begin{equation*}
    \Big(\sqrt{n}(\Bar{\mathbf{X}}-m(\theta_n)),\mathbf{X}^{T}B_{n}\mathbf{X}, {\mathbf X}^TB_n^2{\mathbf X}\Big)\stackrel{d}{\longrightarrow}(W_{\theta_0},(1-m^2(\theta_0))S_0, (1-m^2(\theta_0))T_0).
\end{equation*}

\item[(ii)]
Further we have
$$\lim_{n\to\infty}\Big\{Z_n(\theta_n,{\rm CW})-Z_n(\theta_0,{\rm CW})-\frac{\sqrt{n}m^2(\theta_0)}{2}\Big\}=\frac{R(\theta_0)h^2}{2},$$
  \end{enumerate}
  where $R(\theta_0)$ is as defined in Theorem \ref{test_low}.
  \\
  
    \item[(b)]\textbf{Critical Point}:\label{XBX_a2}: $\theta_0\in \Theta_{2}$. 
    \\
    
    \begin{enumerate}
        \item[(i)] 
  We have \begin{equation*}
    \big(n^{1/4}\Bar{\mathbf{X}}, \mathbf{X}^{T}B_{n}\mathbf{X},\mathbf{X}^{T}B_{n}^{2}\mathbf{X}\big)\stackrel{d}{\longrightarrow} (U_h,S_0,T_0).
\end{equation*}

\item[(ii)]
Further, we have
$$ \lim_{n\to\infty}\{Z_n(1+h_n)-Z_n(1)\}=F(h)-F(0),$$
where $F(.)$ is as defined in \eqref{H_h}.
\end{enumerate}

%
%
%
%
    \end{enumerate}
    \end{lem}

\begin{proof}
\begin{enumerate}
    \item[(a)]
Before we begin the proof, we point out that Lemma \ref{lem:phi1} part (c) implies that by symmetry, $$\phi_n\stackrel{d}{\to}\frac{1}{2}(\delta_{m(\theta_0)}+\delta_{-m(\theta_0)}),$$
    and Proposition \ref{ppn_aux} implies that 
    $$|{\phi}_n-\bar{\mathbf X}|\stackrel{p}{\to}0.$$
  The last two displays together imply $$\P_{n,\theta_n}({\phi_n}<0|\bar{\mathbf X}>0)\to 0,$$
    and so without loss of generality we can interchange between the conditioning events $\bar{\mathbf X}>0$ and $\phi_n>0$. Also, conditional on $\phi_n>0$ we have $\phi_n\stackrel{p}{\to}m(\theta_0)$, from Lemma \ref{lem:phi1} part (c).

\begin{enumerate}
    \item[(i)] Using Proposition \ref{ppn_aux}, conditioning on $\phi_n$, the random variables $(X_1,\ldots,X_n)$ are IID, with
    $$\E(X_1|\phi_n)=\tanh(\theta_n\phi_n)=:\mu_n,\quad Var(X_1|\phi_n)=\text{sech}^2(\theta_n\phi_n).$$ Noting that $B_n{\bf 1}={\bf 0}$, setting ${\bm \mu}_n:=\mu_n{\bf 1}$ one can write
    \begin{align}
    \begin{split}
       & \Big[\sqrt{n}(\bar{\mathbf X}-m(\theta_n)), {\mathbf X}^TB_n{\mathbf X}, {\mathbf X}^TB^2_n{\mathbf X}\Big]\\
      \label{eq:quadratic}  =& \Big[\sqrt{n}(\bar{\mathbf X}-\mu_n),  ({\mathbf X}-{\bm \mu}_n)^T B_n({\mathbf X}-{\bm \mu}_n), ({\mathbf X}-{\bm \mu}_n)^TB^2_n({\mathbf X}-{\bm \mu}_n)\Big]\\
      +&\Big[\sqrt{n}(\mu_n-m(\theta_n)),0,0\Big].
      \end{split}
    \end{align}
    By Lemma \ref{lem:comparison}, conditioning on $\phi_n$, on the event $\phi_n>0$ we get
    \begin{align}\label{eq:quadratic2}
        \begin{split} &\Big[\sqrt{n}(\bar{\mathbf X}-\mu_n),  ({\mathbf X}-\bm{\mu}_n)^T B_n({\mathbf X}-\bm{\mu}_n), ({\mathbf X}-\bm {\mu}_n)^TB^2_n({\mathbf X}-\bm{\mu}_n)\Big]\\
         \stackrel{d}{\to}&\Big[\tau W_0, \tau^2S_0, \tau^2 T_0\Big],
         \end{split}
    \end{align}
    where $$\text{sech}(\theta_n\phi_n)\stackrel{p}{\to} \text{sech}(\theta_0m(\theta_0))=\sqrt{1-m^2(\theta_0)}=:\tau.$$
    Proceeding to analyze the second term in the RHS of \eqref{eq:quadratic}, a one term Taylor's expansion then gives that 
    \begin{align*}
        \sqrt{n}(\mu_n-m(\theta_n))=&\sqrt{n}\Big(\tanh(\theta_n\phi_n)-\tanh(\theta_nm(\theta_n))\Big)\\
        =&\sqrt{n}(\phi_n-m(\theta_n))\theta_n\text{sech}^2(\theta_n\xi_n),
    \end{align*}
    where $\xi_n$ lies between $\phi_n$ and $m(\theta_n)$, and hence converges to $m(\theta_0)$ in probability. Using the above display, conditional on $\phi_n>0$ we have
   \begin{align*}
        \sqrt{n}(\mu_n-m(\theta_n))\stackrel{d}{\to}&\theta_0 \text{sech}^2(\theta_0 m(\theta_0))N\Big(0,\frac{1}{\theta_0-(1-m^2(\theta_0))\theta_0^2}\Big)\\
        =& N\Big(0,\frac{\theta_0(1-m^2(\theta_0))^2}{1-(1-m^2(\theta_0))\theta_0}\Big).
        \end{align*}
        Combining the last display along with \eqref{eq:quadratic} and \eqref{eq:quadratic2}, the conclusion of part (a) follows, on noting that
        $$1-m^2(\theta_0)+\frac{\theta_0(1-m^2(\theta_0))^2}{1-(1-m^2(\theta_0))\theta_0}=\frac{1-m^2(\theta_0)}{1-\theta_0(1-m^2(\theta_0))}=\sigma^2(\theta_0).$$

        \item[(ii)]
        To begin, use Proposition \ref{ppn_aux} to get
        \begin{align}\label{eq:ii0}
        \phi_n=\bar{\mathbf X}+\frac{W_0}{\sqrt{n\theta_n}},
            \end{align}
            where $W_0\sim N(0,1)$ is independent of $\mathbf{X}$. This gives
            \begin{align}\label{eq:ii1}\sqrt{n}\Big(\bar{\mathbf X}-m(\theta_n)\Big)=\sqrt{n}\Big(\phi_n-m(\theta_n)\Big)-\frac{W_0}{\sqrt{n\theta_n}}.
            \end{align}
            which along with part (c) of Lemma \ref{lem:phi1} shows that conditional on $\bar{\mathbf X}>0$ we have
            $$\sqrt{n}\Big(\bar{\mathbf X}-m(\theta_n)\Big)\stackrel{d}{\to}N\Big(0, \frac{1-m^2(\theta_0)}{1-\theta_0(1-m^2(\theta_0))}\Big)=N(0,\sigma^2(\theta_0)).$$
            Thus in turn implies that unconditionally, we have
             $$\sqrt{n}\Big(\bar{\mathbf X}^2-m^2(\theta_n)\Big)\stackrel{d}{\to}N(0,4\sigma^2(\theta_0)m^2(\theta_0)).$$
             Using Proposition \ref{ppn:m}, we then have
             $$\frac{\sqrt{n}}{2}\Big(\bar{\mathbf X}^2-m^2(\theta_0)\Big)\stackrel{d}{\to}N\Big(m^2(\theta_0) \sigma^2(\theta_0) h,m^2(\theta_0)\sigma^2(\theta_0)\Big),$$
which gives
        \begin{align*}
        \frac{Z_n'(\theta_0+h/\sqrt{n}, {\rm CW})}{\sqrt{n}}-\frac{\sqrt{n} m^2(\theta_0)}{2}=&\frac{\sqrt{n}}{2}\Big(\E_{\P_{\theta_n, {\rm CW}}} \bar{\mathbf X}^2-m^2(\theta_0)\Big)\\
        \to &\E N\Big(m^2(\theta_0)\sigma^2(\theta_0)h, m^2(\theta_0)\sigma^2(\theta_0)\Big)^2\\
        =&m^2(\theta_0)\sigma^2(\theta_0)h=R(\theta_0)h.
        \end{align*}
        In the line above, we have used the fact that $\sqrt{n}(\bar{\mathbf X}^2-m^2(\theta_0))$ is uniformly integrable. But this follows on using \eqref{eq:ii1}, along with the fact that $\sqrt{n}(\phi_n-m(\theta_n))$ is uniformly integrable (from Lemma \ref{lem:phi1} part (c)).
        \\
        
      The convergence in the above display holds for all $h$ fixed. Integrating both sides over the interval $[0,h]$ we get
        $$Z_n(\theta_0+h/\sqrt{n},{\rm CW})-Z_n(\theta_0,{\rm CW})-\frac{\sqrt{n}hm^2(\theta_0)}{2}\to \frac{R(\theta_0)h^2}{2},$$
        as desired.
        In the last convergence above, we use the fact that the function $h\mapsto \frac{Z_n'(\theta_0+h/\sqrt{n},{\rm CW})}{\sqrt{n}}$ is monotone, and hence converges uniformly over compact sets, by Proposition \ref{ppn:monotone}. 
        This completes the proof of part (a).
        \\
        
    \end{enumerate}

    \item[(b)]

    \begin{enumerate}
        \item[(i)]
        Again using calculations similar to \eqref{eq:quadratic}, we get
        \begin{align*}
       & \Big[n^{1/4}\bar{\mathbf X}, {\mathbf X}^TB_n{\mathbf X}, {\mathbf X}^TB^2_n{\mathbf X}\Big]\\
      =& \Big[n^{1/4}(\bar{\mathbf X}-\mu_n),  ({\mathbf X}-{\bm \mu}_n)^T B_n({\mathbf X}-{\bm \mu}_n), ({\mathbf X}-{\bm \mu}_n)^TB^2_n({\mathbf X}-{\bm \mu}_n)\Big]+[n^{1/4}\mu_n,0,0].
        \end{align*}
        Conditioning on $\phi_n$, using Proposition \ref{ppn_aux} and Lemma \ref{lem:comparison} we have
        $$\Big[n^{1/4}(\bar{\mathbf X}-\mu_n),  ({\mathbf X}-{\bm \mu}_n)^T B_n({\mathbf X}-{\bm \mu}_n), ({\mathbf X}-{\bm \mu}_n)^TB^2_n({\mathbf X}-{\bm \mu}_n)\Big]\stackrel{d}{\to}[0, S_0,T_0],$$
        where we use the fact
        $$Var(X_1|\phi)=\text{sech}^2(\theta_n\phi_n)\stackrel{p}{\to}1.$$
        Also, an application of delta theorem along with Lemma \ref{lem:phi1} part (b) gives
        $$n^{1/4}\mu_n=n^{1/4}\tanh(\phi_n)\stackrel{d}{\to}U_{1.h}.$$
        Combining the above, it follows that
        $$\Big[n^{1/4}\bar{\mathbf X}, {\mathbf X}^TB_n{\mathbf X}, {\mathbf X}^TB^2_n{\mathbf X}\Big]\stackrel{d}{\to}[U_h, S_0, T_0]$$
        as desired.
        
        \item[(ii)]
        
        Using \eqref{eq:ii0} we get
        \begin{align*}
           n^{1/4}\bar{\mathbf X}=n^{1/4}\phi_n-n^{1/4}\frac{W_0}{\sqrt{n\theta_0}}\stackrel{d}{\to}U_h,
           \end{align*}
           where we have used Lemma \ref{lem:phi1} part (b). 
           This in turn gives,
        $$\sqrt{n} \bar{\mathbf X}^2\stackrel{d}{\to}U_h^2.$$
       and so
        \begin{align*}
            \frac{Z_n'(\theta_n)}{\sqrt{n}}=\frac{\sqrt{n}}{2}\E_{\P_{\theta_n, {\rm CW}}}\bar{\mathbf X}^2\to \frac{1}{2}\E U_h^2=F'(h).
        \end{align*}
        Again in the last step we use the fact that $\sqrt{n}\bar{X}^2$ is uniformly integrable, which follows from \eqref{eq:ii0} and Lemma \ref{lem:phi1} part (b). The desired conclusion again follows on integrating the above display over $[0,h]$, on noting that the above convergence is uniform on compact sets, by Proposition \ref{ppn:monotone}.

        \end{enumerate}

\end{enumerate}

\end{proof}

    \begin{lem}\label{Curie_weiss_LE}
    Suppose the matrix $Q_n$ satisfies \eqref{RIC}, \eqref{CWC}, \eqref{eq:frob} and \eqref{eq:cut}  for some $C_W, \kappa\in (0,\infty)$ and $f\in \mathcal{W}$. Let $\theta_0>0,h\in \R$, and $\theta_n=\theta_0+\frac{h}{\sqrt{n}}$ be as in definition \ref{def:thetan}. Then the following conclusions hold for $\theta\in \Theta_1\cup \Theta_2$.
    \begin{enumerate}
     \item[(a)]\label{XBX_b} 
    \begin{equation*}
    \lim\limits_{n\rightarrow\infty}Z_{n}(\theta_{n},Q_{n})-Z_{n}(\theta_{n},{\rm CW})=C(\theta_0),
    \end{equation*}
    where
    \begin{equation*}
    C(\theta_0):=-\frac{1}{2}\theta(1-m^2)+\frac{\kappa}{4}\theta_0^{2}(1-m^{2})^{2}+\frac{1}{2}\sum_{i=2}^{\infty}\Big[\log\big(1-\theta_0(1-m^2)\lambda_i\big)+\lambda_i\theta_0(1-m^2)\Big].
\end{equation*}
\item[(b)]\label{XBX_c} The probability measures $\P_{\theta_{n},Q_{n}}$ and $\P_{\theta_{n},{\rm CW}}$ are mutually contiguous.
\end{enumerate}
\end{lem}

\begin{proof}
\begin{enumerate}
\item[(a)]
To begin, note that
\begin{align*}
    \exp\Big(Z_n(\theta_n,Q_n)-Z_n(\theta_n,{\rm CW})\Big)=&\frac{\sum_{{\bf x}\in \{-1,1\}^n} e^{\frac{\theta_n}{2}{\mathbf X}^T Q_n{\mathbf X}}}{\sum_{{\bf x}\in \{-1,1\}^n} e^{\frac{\theta_n}{2}n\bar{\mathbf X}^2}}\\
    =&\E_{\theta_n, {\rm CW}} e^{\frac{\theta_n}{2}{\mathbf X}^TB_n{\mathbf X}}.
\end{align*}
It follows from Lemma \ref{XBX_CW} that if ${\mathbf X}\sim \P_{\theta_n,{\rm CW}}$, then for all $\theta_0\in \Theta_0$ we have
$${\mathbf X}^TB_n{\mathbf X}\stackrel{d}{\to}(1-m^2(\theta_0))S_0\Rightarrow e^{\frac{\theta_n}{2}}{\mathbf X}^TB_n{\mathbf X}\stackrel{d}{\to} e^{\frac{\theta_0(1-m^2(\theta_0))}{2}S_0}.$$
Assume now that there exists $\delta>0$ such that
\begin{align}\label{eq:ui_final}
\E_{\theta_n, {\rm CW}} e^{\frac{(1+\delta)\theta_n}{2}{\mathbf X}^TB_n{\mathbf X}}<\infty.
\end{align}
Uniform integrability then gives
$$\E_{\theta_n, {\rm CW}}e^{\frac{\theta_n}{2}{\mathbf X}^TB_n{\mathbf X}}\to \E e^{\frac{\theta_0(1-m^2(\theta_0))}{2}S_0}, $$
which equals $C(\theta_0)$ using the formula for $S_0$ (see \eqref{eq:ss}).
\\

It thus remains to verify \eqref{eq:ui_final}. To this effect, note that
\begin{align*}
    \E_{\theta_n, {\rm CW}} e^{\frac{\theta_n}{2}{\mathbf X}^TB_n{\mathbf X}}=\E_{\theta_n, {\rm CW}}\Big( e^{\frac{\theta_n}{2}({\mathbf X}-{\bm \mu}_n)^TB_n({\mathbf X}-{\bm \mu}_n)}\Big|\phi_n\Big),
\end{align*}
where ${\bm \mu}_n=\mu_n{\bf 1}$ with $\mu_n=\tanh(\theta_n\phi_n)$, as in the proof of Lemma \ref{XBX_CW}. Invoking Lemma \ref{ppn_aux}, we have that given $\phi_n$ the random variables $(X_1,\ldots,X_n)$ are IID with mean $\mu_n$. Also, setting 
\begin{eqnarray*}s_\mu:=&\frac{2\mu}{\log(1+\mu)-\log(1-\mu)}&\text{ if }\mu\ne 0,\\
=&1&\text{ if }\mu=0
\end{eqnarray*}
we have that $s_.$ is a strictly positive continuous even function, with
$s_{\mu_n}\stackrel{p}{\to}s_{m(\theta_0)}=\frac{1}{\theta_0}$
for all $\theta_0\in \Theta_1\cup \Theta_2$,
where the last equality uses the fact that $m(\theta_0)=\tanh(\theta_0m(\theta_0))$. Since $$\limsup_{n\to\infty}\lambda_1(B_n)=\lambda_2<1$$ by \eqref{eq:cut}, there exists $\delta>0$ such that on the set $||\mu_n|-m(\theta_0)|>\delta$ we have
$$\limsup_{n\to\infty}\theta_n \lambda_1(B_n)s_{\mu_n}<1.$$
Thus
using \cite[Proposition 4.1]{deb2020fluctuations} with
$$N=n, \quad D_N(i,j)=\theta_n B_n(i,j),\quad c_i=0,$$
we get the existence of a constant $C$ free of $n$ such that on the set $||\mu_n|-m|>\delta$ we have
$$\log\E_{\theta_n, {\rm CW}}\Big( e^{\frac{\theta_n}{2}({\mathbf X}-{\bm \mu}_n)^TB_n({\mathbf X}-{\bm \mu}_n)}\Big|\phi_n\Big)\le C.$$
To complete the proof of \eqref{eq:ui_final}, it suffices to show that 
$$\limsup_{n\to\infty}\E_{\theta_n, {\rm CW}} e^{\frac{\theta_n}{2}{\mathbf X}^TB_n{\mathbf X}}1\{|\mu_n|-m(\theta_0)|>\delta\}<\infty,$$
for some $\delta>0$.
But this follows from \cite[Lemma 4.2]{deb2020fluctuations} on using
$$N=n, \quad V_N=\theta_n{\mathbf X}^TB_n{\mathbf X}.$$
We note that the proof of \cite[Lemma 4.2]{deb2020fluctuations} goes through verbatim if $\theta=\theta_n$ depends on $n$, even though the lemma is stated for a fixed $\theta_0$.
\\

    \item[(b)]
   The likelihood ratio between $\P_{\theta_n,Q_n}$ and $\P_{\theta_n}$ is given by
   \begin{align*}
    \frac{d\P_{\theta_n,Q_n}}{d\P_{\theta_n, {\rm CW}}}({\mathbf X})=\exp\Big(\frac{\theta_n}{2}{\mathbf X}^TB_n{\mathbf X}-Z_n(\theta_n,Q_n)+Z_n(\theta_n,{\rm CW})\Big).
       \end{align*}
       Using Lemma \ref{XBX_CW}, for all $\theta_0\in \Theta$ we have
       $${\mathbf X}^TB_n{\mathbf X}\stackrel{d}{\to}(1-m^2(\theta_0))S_0.$$
   Also, using part (a) we have $$Z_n(\theta_n,Q_n)-Z_n(\theta_n,{\rm CW})\to \log \E e^{\frac{\theta_0(1-m^2(\theta_0))}{2}S_0}.$$
   Combining, if ${\mathbf X}\sim \P_{\theta_n, {\rm CW}}$, then we have
   \begin{align*}
       \frac{d\P_{\theta_n,Q_n}}{d\P_{\theta_n, {\rm CW}}}({\mathbf X})\stackrel{d}{\to} \frac{e^{\frac{\theta_0(1-m^2(\theta_0))}{2}S_0}}{\E e^{\frac{\theta_0(1-m^2(\theta_0))}{2}S_0}}.
   \end{align*}
   Since the limiting random variable in the above display is strictly positive and has mean 1, mutual contiguity follows by Le-Cam's first lemma.
    
\end{enumerate}
\end{proof}

\subsection{Proof of Lemma \ref{lem:mean} and Lemma \ref{lem:normalizing_ising}}\label{sec:C}

Using Lemma \ref{XBX_CW} and Lemma \ref{Curie_weiss_LE} , we now prove Lemma \ref{lem:mean} and Lemma \ref{lem:normalizing_ising}, which were stated in the main draft, and were used to prove our main results.

\begin{proof}[Proof of Lemma \ref{lem:mean}]

To begin, note that in all regimes of $\theta$, the following hold:
\begin{itemize}
    \item The log likelihood ratio \begin{align}\label{eq:ll}
    \log\frac{d\P_{\theta_n,Q_n}}{d\P_{\theta_n, {\rm CW}}}({\mathbf X})=\frac{\theta_n}{2}{\mathbf X}^TB_n{\mathbf X}-Z_n(\theta_n,Q_n)+Z_n(\theta_n,{\rm CW})
    \end{align}
    is a function of ${\mathbf X}^TB_n{\mathbf X}$.
    
    \item
    The asymptotic non degenerate limiting distribution of $\bar{\mathbf X}$ is jointly independent of $\mathbf{X}^TB_n\mathbf{X}$ and ${\mathbf{X}}^TB_n^2{\mathbf X}$ (this follows from Lemma \ref{XBX_CW}).
    
    \item
    The two measures $\P_{\theta_n,Q_n}$ and $\P_{\theta_n, {\rm CW}}$ are mutually contiguous (this follows from Lemma \ref{Curie_weiss_LE} part (b)).
\end{itemize}
It thus follows from Le-Cam's third lemma that the asymptotic non degenerate distribution of $\bar{\mathbf X}$ under $\P_{\theta_n,Q_n}$ and $\P_{\theta_n, {\rm CW}}$ are the same, and is asymptotically independent of the joint distribution of $({\mathbf X}^TB_n{\mathbf X}, {\mathbf X}^TB_n^2 {\mathbf X})$ under both models.
\\

To complete the proof of Lemma \ref{lem:mean}, it then suffices to show that for $\theta_0\in\Theta_1\cup \Theta_2$, under $\P_{\theta_n,Q_n}$ we have 
\begin{align}\label{eq:all_regimes2}
    (\mathbf{X}^TB_n\mathbf{X}, {\mathbf{X}}^TB_n^2{\mathbf X})\stackrel{d}{\to}(S_{\theta_0},T_{\theta_0}).
\end{align}
To show this, first note that under $\P_{\theta_n, {\rm CW}}$ we have
\begin{align*}
   & \left[\mathbf{X}^TB_n\mathbf{X}, {\mathbf{X}}^TB_n^2{\mathbf X},  \log\frac{d\P_{\theta_n,Q_n}}{d\P_{\theta_n, {\rm CW}}}({\mathbf X}) \right]\\
   =& \left[\mathbf{X}^TB_n\mathbf{X}, {\mathbf{X}}^TB_n^2{\mathbf X},  \frac{\theta_n}{2}{\mathbf X}^TB_n{\mathbf X}-Z_n(\theta_n,Q_n)+Z_n(\theta_n,{\rm CW})\right]\\
   \stackrel{d}{\to}&\left[\Big(1-m^2(\theta_0)\Big)S_0, \Big(1-m^2(\theta_0)\Big)T_0, \frac{\theta_0\Big(1-m^2(\theta_0)\Big)}{2}S_0-C(\theta_0)\right],
\end{align*}
 and we have used \eqref{eq:ll} in the first step, and Lemma \ref{XBX_CW} and Lemma \ref{Curie_weiss_LE} part (a) in the second step (and $C(\theta_0)$ is defined in Lemma \ref{Curie_weiss_LE} part (a)).
Using mutual contiguity, it follows that under $\P_{\theta_n,Q_n}$, we have
 $$(\mathbf{X}^TB_n\mathbf{X}, {\mathbf{X}}^TB_n^2{\mathbf X})\stackrel{d}{\to}(S',T'),
$$
where $(S',T')$ is a bi-variate random vector with characteristic function
\begin{align*}&\E e^{i(sS'+tT')}\\
&= \E \exp\left\{(1-m^2(\theta_0))i (sS_{0}+tT_{0})+\frac{\theta_0(1-m^2(\theta_0))}{2}S_0-C(\theta_0)\right\}\\
  &=e^{-C(\theta_0)}\E \exp \left\{(1-m^{2}(\theta_0))\Big(is+\frac{\theta_0}{2}\Big)S_0+it T_0\right\}\\
    &=e^{-C(\theta_0)}\E \exp\left\{(1-m^{2}(\theta_0))\Big(is+\frac{\theta_0}{2}\Big)(\sum\limits_{j=2}^{\infty}\lambda_{j}(Y_{j}-1)-1+W^*)+it(1-m^{2}(\theta_0))\big(\sum\limits_{j=2}^{\infty}\lambda_{j}^{2}Y_{j}+\kappa\big)\right\} \\
    &=e^{-C(\theta_0)}\exp\Big\{-(1-m^2(\theta_0))\Big(is+\frac{\theta_0}{2}\Big)+i\kappa t(1-m^2(\theta_0))\Big\}\E \exp\left\{(1-m^2(\theta_0))\big(is+\frac{\theta_0}{2}\big)W^*\right\}\\
    & \prod_{j=2}^\infty \exp\Big\{-(1-m^2(\theta_0))\Big(is+\frac{\theta_0}{2}\Big)\lambda_j\Big\}\E \exp\left\{ (1-m^2(\theta_0))\Big(is+\frac{\theta_0}{2}\Big)\lambda_j Y_j+(1-m^2(\theta_0))t\lambda_j^2 Y_j\right\}\\
    &=e^{-C(\theta_0)}\exp\left\{(1-m^{2}(\theta_0))\Big(it\kappa-is-\frac{\theta_0}{2}\Big)+(1-m^{2}(\theta_0))^{2}\Big(is+\frac{\theta_0}{2}\Big)^{2}\kappa\right\}\\
    &\prod\limits_{j=2}^{\infty}\frac{\exp\Big\{-(1-m^2(\theta_0))\Big(is+\frac{\theta_0}{2}\Big)\lambda_j\Big\}}{\sqrt{1-(1-m^{2}(\theta_0))\big(\theta_0\lambda_{j}+2i\lambda_{j}s+2i\lambda_{j}^{2}t)}}\\
    &= e^{(1-m^{2}(\theta_0))(it\kappa-is)+(1-m^{2}(\theta_0))^{2}(i\theta_0 s\kappa-s^{2}\kappa)}\prod\limits_{j=2}^{\infty}\frac{e^{-(1-m^2(\theta_0))\Big(is+\frac{\theta_0}{2}\Big)\lambda_j}}{\sqrt{1-(1-m^{2}(\theta_0))\Big(1-2i(1-m^{2}(\theta_0))\frac{\lambda_{j}s+\lambda_{j}^{2}t}{1-\theta_0(1-m^{2}(\theta_0))\lambda_{j}}\Big)}},
    \end{align*}
    where in the last step we use the formula for $C(\theta_0)$ from Lemma \ref{Curie_weiss_LE} part (a). The last display above can be checked to be the joint characteristic function of  \begin{align*}
    (1-m^2(\theta_0))&\bigg[\sum\limits_{j=2}^{\infty}\lambda_{j}\Big(\frac{Y_{j}}{1-\theta_0(1-m^{2}(\theta_0))\lambda_{j}}-1\Big)-1+(1-m^{2}(\theta_0))\theta_0\kappa+W^*,\\
    &\sum\limits_{j=2}^{\infty}\frac{\lambda_{j}^{2}Y_{j}}{1-\theta_0(1-m^{2}(\theta_0))\lambda_{j}}+\kappa\bigg],
    \end{align*}
    and so the proof of \eqref{eq:all_regimes2} is complete.

\end{proof}

\begin{proof}[Proof of Lemma \ref{lem:normalizing_ising}]
Using Lemma \ref{Curie_weiss_LE} part (a), it suffices to replace $Z_n(\theta_n,Q_n)-Z_n(\theta_0,Q_n)$ by $Z_n(\theta_n,{\rm CW})-Z_n(\theta_0,{\rm CW})$. But this is exactly what was proved in Lemma \ref{XBX_CW} part (a)(ii) and part (b)(ii).

 \end{proof}

	\end{document}